\documentclass[12pt]{amsart}
\usepackage{amsmath}
\usepackage{amsxtra}
\usepackage{amscd}
\usepackage{amsthm}
\usepackage{amsfonts}
\usepackage{amssymb}
\usepackage {pstricks}
\usepackage{pstricks,pst-node}
\usepackage[all]{xy}

\usepackage{amsbsy}
\usepackage{mathrsfs}

\usepackage{verbatim}

\usepackage{color}
\usepackage{indentfirst}
\usepackage{latexsym,bm}
\usepackage{graphicx}
\usepackage{cases}
\usepackage{pifont}
\usepackage{txfonts}
\usepackage{xcolor}

\usepackage{longtable}
\usepackage{setspace}

\usepackage{cases}
\usepackage{multirow}
\usepackage{array}
\usepackage{makecell}
\usepackage{supertabular}
\usepackage{booktabs}
\newcommand{\longsquiggly}{\xymatrix{{}\ar@{~>}[r]&{}}}

\newtheorem{theorem}{Theorem}[section]

\newtheorem{lemma}[theorem]{Lemma}

\newtheorem{corollary}[theorem]{Corollary}
\newtheorem{proposition}[theorem]{Proposition}

\theoremstyle{definition}
\newtheorem{definition}[theorem]{Definition}

\theoremstyle{remark}
\newtheorem{remark}[theorem]{Remark}

\numberwithin{equation}{section}
\newtheorem{case}{Case}

\def\1{1\kern-.3em1}

\newcommand{\N}{{\mathbb N}}
\newcommand{\Z}{{\mathbb Z}}

\newcommand{\R}{{\mathbb R}}
\newcommand{\C}{{\mathbb C}}

\newcommand{\U}{{\rm U}}

\newcommand{\osp}{{\rm\mathfrak{osp}}}
\newcommand{\Sl}{{\rm\mathfrak{sl}}}

\newcommand{\UU}{{\mathfrak U}}
\newcommand{\g}{{\mathfrak g}}

\newcommand{\Uq}{{\rm{U}_q}}

\newcommand{\id}{{\rm{id}}}

\newcommand{\cT}{{\mathscr T}}
\newcommand{\cH}{{\mathscr H}}

\newcommand{\cG}{{\mathcal G}}

\newcommand{\cJ}{{\mathcal J}}

\newcommand{\so}{{\rm\mathfrak{so}}}
\newcommand{\ad}{{\mbox{Ad}}}

\newcommand{\mb}[1]{\mathbb{#1}}

\newcommand{\al}{\alpha}

\newcommand{\de}{\delta}
\newcommand{\De}{\Delta}

\newcommand{\ve}{\varepsilonup}

\newcommand{\lb}{\left(}
\newcommand{\rb}{\right)}

\newcommand{\td}{\tilde{\Delta}}
\newcommand{\ts}{\tilde{S}}

\allowdisplaybreaks

\newcommand{\cE}{{\mathcal E}}

\begin{document}
\normalfont
\sloppy

\title[Quantum correspondences]{Quantum correspondences of affine Lie superalgebras}
\author{Ying Xu}
\author{R. B. Zhang}
\address[Xu]{School of Mathematics, Hefei University of Technology, Anhui Province, 230009, China}
\address[Xu, Zhang]{School of Mathematics and Statistics,
University of Sydney, NSW 2006, Australia}
\email{yingxu@maths.usyd.edu.au, ruibin.zhang@sydney.edu.au}
\begin{abstract}
There is a surprising isomorphism between the quantised universal enveloping algebras of $\osp(1|2n)$ and $\so(2n+1)$.
This same isomorphism emerged in recent work of Mikhaylov and Witten
in the context of string theory as a $T$-duality composed with an S-duality.
We construct similar Hopf superalgebra isomorphisms for families of pairs of quantum affine superalgebras. An immediate consequence is that the representation categories of the quantum affine superalgebras in each pair are equivalent as strict tensor categories.
\end{abstract}
\maketitle


\section{Introduction}\label{sect:introduction}
In the early 80s, Rittenberg and Scheunert  \cite{RS} observed a remarkable connection between the orthogonal Lie algebra $\so(2n+1)$ and
the orthosymplectic Lie superalgebra $\osp(1|2n)$: there is a one-to-one
correspondence between the finite dimensional  representations of
$\osp(1|2n)$ and the tensorial representations of $\so(2n+1)$,
and the central characters of the two algebras
in the corresponding simple modules are the same.
This remained a mystery until
quantum supergroups \cite{BGZ, Y,  ZGB,Z93, Z98}
came to the scene.
It was discovered in \cite{Z3} that the quantised universal enveloping
algebras of these Lie (super)algebras are essentially isomorphic, and the
Rittenberg-Scheunert correspondence is a consequence of this isomorphism
in the semiclassical limits.

The isomorphism between the quantum (super)groups of $\osp(1|2n)$ and  $\so(2n+1)$
was used by Lanzmann to great effect \cite{LE} in the study of primitive ideals of $\U(\osp(1|2n))$.
By relating them to the primitive ideals of $\U(\so(2n+1))$ via the isomorphism, 
he drastically simplified the proofs of the results first obtained in \cite{GL}.

The same isomorphism emerged in very recent work of Mikhaylov and Witten \cite{MW} on quantum Chern-Simons theories.
The authors gave a description of Chern-Simons theories with super gauge groups
in terms of systems of D3-branes ending on 2-sides of an NS5-brane.
A T-duality composed with an S-duality of the brane systems interchanges
the corresponding quantum Chern-Simons theories
with gauge groups $\osp(1|2n)$ and $\so(2n+1)$ respectively.
The strong-weak coupling transformation procured by the T-duality corresponds
precisely to the interchange $q\leftrightarrow  -q$ in the quantum group context \cite{Z3}. Furthermore,  a similar duality between
quantum Chern-Simons theories
with gauge groups $\osp(2m+1|2n)$ and  $\osp(2n+1|2m)$ was constructed for arbitrary $m$ and $n$ in \cite{MW}.

The aim of the present paper is to give a catalogue of the isomorphisms analogous to
that between the quantum (super)groups of $\osp(1|2n)$ and  $\so(2n+1)$. The main  results are summarised in Theorem \ref{them:hopf-connect}.
To explain the content of the theorem, we need several conceptual constructs.

Let $\g$ be a Lie superalgebra or an affine Lie superalgebra in the classical series.
Fixing a fundamental system $\Pi$
of simple roots (Definition \ref {defi:quantised})  corresponding to an arbitrary choice of Borel subalgebra,
we let $\U_{q}(\g, \Pi)$ be the quantised universal enveloping superalgebra of $\g$ with respect to this fundamental system.
We note that in general, the quantised universal enveloping superalgebras $\U_{q}(\g, \Pi)$ associated with different fundamental systems $\Pi$
are not isomorphic as Hopf superalgebras.

Corresponding to each $\alpha_i$ in $\Pi$, we introduce a $\Z_2$ group generated by $\sigma_i$ such that $\sigma_i^2=1$, and let $G$ be the direct product of all
the $\Z_2$ groups, i.e.,  $\mathrm{G}:=\Z_2\times \dots\times \Z_2$ ($|\Pi|$ copies).
There is a natural action of  $\mathrm{G}$ on $\U_q(\g, \Pi)$
such that there exists an element $u\in \mathrm{G}$ which implements the  $\Z_2$-grading.
Let $\UU_{q}(\g, \Pi)=\U_{q}(\g, \Pi)\sharp\C\mathrm{G}$ be the smash product (Definition \ref {defi:smash product}), which
has a natural Hopf superalgebra structure by Proposition \ref{prop:hopf}.

The following result is a part of Theorem \ref{them:hopf-connect}.
\begin{theorem}\label{thm:main-quan}
Let $(\g, \g')$ be a pair of Lie superalgebras or affine Lie superalgebras listed in any column of Table \ref{tbl:q-correspond} (where $m+n>0$).

\vspace{-1mm}
\begin{table}[!hbp]
\renewcommand{\arraystretch}{1.2}
\caption{Quantum correspondences}
\label{tbl:q-correspond}
\begin{tabular}{c|c|c|c}
\hline
$\g$ & $\osp(2m+1|2n)$ & $\Sl(2m+1|2n)^{(2)}$  & $\osp(2m+2|2n)^{(2)}$ \\
\hline
$\g' $ & $\osp(2n+1|2m)$ & $\osp(2n+1|2m)^{(1)}$ & $\osp(2n+2|2m)^{(2)}$  \\
\hline
\end{tabular}
\end{table}

\vspace{-1mm}
\noindent
For any chosen fundamental system $\Pi$ of $\g$ and the corresponding  fundamental system $\Pi'=\phi(\Pi)$ of $\g'$ (see Lemma \ref{lem:connection phi}),
there exists an isomorphism of associative algebras
\begin{eqnarray*}
\UU_{-q}(\g', \Pi')\stackrel{\cong}{\longrightarrow}\UU_{q}(\g, \Pi),
\end{eqnarray*}
which is defined explicitly in Theorem \ref{thm:iso-main}.
\end{theorem}

For the special pairs $(\g, \g')$ with either $\g$ or $\g'$ being an ordinary affine Lie algebra,
Theorem \ref{thm:main-quan} was proved in \cite{Z3, Z2}.  It was quite a surprise that such isomorphisms exist as the root systems of any pair $\g$ and $\g'$ are very different.

As it stands, the isomorphism in Theorem \ref{thm:main-quan}
does not preserve the $\Z_2$-gradings, thus can not be a superalgebra isomorphism, let alone a Hopf superalgebra isomorphism.
However, it relates the Hopf superalgebra structures
at a more fundamental level.

As advocated by Majid  and others (see e.g., \cite[Chapter 10.1]{Ma} and \cite{AEG}), one should place Hopf superalgebras in the context of braided tensor categories.
The category of vector superspaces can be considered as the tensor category of representations of the group algebra of $\Z_2$ regarded as a  triangular Hopf algebra.
A Hopf superalgebra is then a Hopf algebra in this category.
Given a Hopf superalgebra, one may change the $\Z_2$-action to obtain a new Hopf superalgebra with the same underlying associative algebra structure, but
a new co-algebra structure and different $\Z_2$-grading.  We refer to this process as a {\em picture change} (see Definition \ref{def:PC}), which is also loosely known as {\em bosonisation} in the literature   \cite{Ma} (see Remark \ref{rem:bosonisation}).
An important fact is that the representation category of
the new Hopf superalgebra is
equivalent to that of the original Hopf superalgebra as strict tensor category (see  \cite[ Chapter 10.1]{Ma} and \cite[Theorem 3.1.1]{AEG}).

Corresponding to any pair $(\g, \g')$ of (affine) Lie superalgebras from Theorem \ref{thm:main-quan}, we have the Hopf superalgebras $\UU_q(\g, \Pi)$ and
$\UU_{-q}(\g', \Pi')$ with the standard Hopf structures.
Denote by $\Delta$, $\epsilon$ and $S$ the co-multiplication, co-unit and antipode of  $\UU_q(\g, \Pi)$ respectively.
We apply an appropriate picture change
to $(\UU_q(\g, \Pi), \Delta, \epsilon, S)$ to obtain a new Hopf superalgebra
$(\UU_q(\g, \Pi), \tilde\Delta, \epsilon, \tilde{S})$, where $\UU_q(\g, \Pi)$ has acquired
a new $\Z_2$-grading. Relative to this $\Z_2$-grading, the map of Theorem \ref{thm:main-quan} becomes an isomorphism of superalgebras; see
Corollary \ref{cor:alg-iso}.

However,  $\UU_{-q}(\g', \Pi')$ and $(\UU_q(\g, \Pi), \tilde\Delta, \epsilon, \tilde{S})$  as Hopf superalgebras are still different.
To relate them, we introduce another ingredient,
Drinfeld twists \cite{D2, R}, which is used for changing the co-algebraic structures. We construct a Drinfeld twist $\cJ$ for $(\UU_q(\g, \Pi), \tilde\Delta, \epsilon, \tilde{S})$ in Lemma \ref{lem:j-properties},  and use it to twist the Hopf superalgebra in the way described in Section \ref{sect:twisting}.
This gives rise to another Hopf superalgebra $(\UU_q(\g, \Pi),  \tilde\Delta^{\cJ}, \epsilon, \tilde{S}^{\cJ})$, see Lemma \ref{lem:tilde-Delta}.

\begin{theorem}\label{them:hopf-connect}
Let $(\g, \g')$ be a pair of (affine) Lie superalgebras  in Theorem \ref{thm:main-quan}. Then the quantum (affine) superalgebra $\UU_{-q}(\g',\Pi')$ with the standard Hopf superalgebra structure is isomorphic to $(\UU_q(\g, \Pi), \tilde{\Delta}^\cJ, \epsilon, \tilde{S}^\cJ)$.
\end{theorem}
We comment that  even though
Theorem \ref{thm:main-quan} was partially known \cite{Z3, Z2} before, the Hopf superalgebra isomorphism of Theorem \ref{them:hopf-connect}
is new in all cases.
In general, the quantised universal enveloping superalgebras $\Uq(\g, \Pi)$  corresponding to different fundamental systems are non-isomorphic as Hopf superalgebras, thus the isomorphism of Theorem \ref{them:hopf-connect} depends on the fundamental systems nontrivially.

\begin{definition}\label{def:correspond} We call the Hopf superalgebra isomorphism
of Theorem \ref{them:hopf-connect}  a quantum correspondence between the (affine) Lie superalgebras $\g$ and $\g'$.
\end{definition}
The following result is a consequence of the quantum correspondence
and some general facts (see Theorem \ref{thm:PC}) on Hopf superalgebras.
\begin{theorem}\label{thm:tensor-equiv} Let  $(\g, \g')$ be
any pair  of (affine) Lie superalgebras in Theorem \ref{thm:main-quan}.
For any fundamental system $\Pi$ of $\g$ and the corresponding  fundamental system $\Pi'=\phi(\Pi)$ of $\g'$,  the representation categories of
the Hopf superalgebras $\UU_{q}(\g, \Pi)$ and $\UU_{-q}(\g', \Pi')$ are
equivalent as strict tensor categories.
\end{theorem}

The remainder of the paper is devoted to the proof of
Theorem \ref{them:hopf-connect}. All notions required, including those used in the discussion above,  will be carefully explained.

Results of the present paper have been applied to construct Drinfeld realisations \cite{XZ}, vertex operator representations, and finite dimensional representations
of classes of quantum affine superalgebras.

 \section{Quantised universal enveloping superalgebras}

\subsection{Root Systems}\label{sect:roots}

We begin with a description of the root data of the classical series of Lie superalgebras and the related twisted and untwisted affine Lie superalgebras.

For any given pair of nonnegative integers $k$ and $l$, we let $\cE(k|l)$ be the $(k+l)$-dimensional vector space over $\R$
with a basis consisting  of elements $\varepsilon_i$ ($i=1, 2, \dots, k$)
and $\delta_\nu$ ($\nu=1, 2, \dots, l$). We endow  $\cE(k|l)$ with a symmetric non-degenerate bi-linear form
\begin{equation}\label{eq:bilinear form}
\begin{aligned}
(\varepsilon_i, \varepsilon_j)=(-1)^{\theta}\delta_{i j}, \quad (\delta_\mu, \delta_\nu)=-(-1)^{\theta}\delta_{\mu \nu}, \quad
(\varepsilon_i, \delta_\mu)=(\delta_\mu, \varepsilon_i)=0,
\end{aligned}
\end{equation}
where $\theta$ is $0$ or $1$ which will be fixed in the following way.
Call an order of the basis elements \emph{admissible} if $\varepsilon_i$ appears before
$\varepsilon_{i+1}$ for all $i$,  and $\delta_\nu$ before $\delta_{\nu+1}$
for all $\nu$. Fix an admissible order and denote by $({\mathcal E}_1, {\mathcal E}_2, \dots, {\mathcal E}_{k+l})$ the ordered basis of $\cE(k|l)$. Then we choose $\theta$ so that
$
({\mathcal E}_1, {\mathcal E}_1)=1.
$

Let $\g$ be either a special linear or orthosymplectic Lie superalgebra. Then the set $\Phi$ of roots of $\g$ can be realized as a subset
of $\cE(k|l)$ for appropriate $k$ and $l$, where we will take the $k$ and $l$ to be the smallest possible.
We will call $\cE(k|l)$ the ambient space of $\Phi$.

Each choice of a Borel subalgebra corresponds to a choice of positive roots, and hence a fundamental system $\Pi=\{\alpha_1, \alpha_2, \dots, \alpha_r\}$ of simple roots, where $r$ is the rank of $\g$.
The Weyl group conjugacy classes of Borel subalgebras correspond bijectively to the admissible ordered bases of the ambient space.

The root data of the classical series of simple Lie superalgebras can be described as in Table \ref{table:classical}, where the ambient space of $\Phi$ is $\cE(m|n)$ in each case. Now $\Phi\subset \cE(m|n)_{\Z}=\sum_{a=1}^{m+n}\Z\cE_a$. Define a map $\chi:\cE(m|n))_{\Z} \longrightarrow \Z$ such that $\chi(v)=\sum_{\mu=1}^{n} b_{\mu}$ for any $v=\sum_{i=1}^m a_i\varepsilon_i+\sum_{\mu=1}^n b_{\mu}\delta_{\mu}$. Then a root $\beta\in\Phi$ is even if $\chi(\beta)$ is even, and odd otherwise.

\vspace{-2mm}
\begin{table}[h]
\renewcommand{\arraystretch}{1.1}
\centering
\caption{Classical series of Lie superalgebras}
\label{table:classical}
\begin{tabular}{c|c}
\hline
$\g$ &\text{simple roots}\\
\hline
${\Sl}(m|n)$ & $\alpha_i={\mathcal E}_i - {\mathcal E}_{i+1},\ \ 1\le i< m+n$\\
\hline
${\osp}(2m+1|2n)$ & $\alpha_i={\mathcal E}_i - {\mathcal E}_{i+1}, \ \ 1\le i<m+n, \quad
  \alpha_{m+n}= {\mathcal E}_{ m+n}$\\
  \hline
\multirow{3}{*}{${\osp}(2m|2n)$} & $\alpha_i={\mathcal E}_i - {\mathcal E}_{i+1}, \ \ 1\le i<m+n, $ \\
& $\alpha_{m+n}=\begin{cases}
{\mathcal E}_{m+n-1}+{\mathcal E}_{m+n}, &\text{ if ${\mathcal E}_{m+n}=\varepsilon_{m}$},\\
2{\mathcal E}_{m+n}, &\text{ if ${\mathcal E}_{m+n}=\delta_{n}$}.
\end{cases}$\\
\hline
\end{tabular}
\end{table}

\vspace{-4mm}
\begin{remark}
The Lie superalgebras  $\osp(m|n)$ and $\Sl(m|n)$ reduce to ordinary Lie algebras
if $m=0$ or $n=0$.
Also, $\Sl(m|m)$ contains the ideal $\C 1_{2m}$, and $\Sl(m|m)/\C1_{2m}$ is simple.
\end{remark}

In order to describe the root data of untwisted and twisted affine Lie superalgebras of the classical series of Lie superalgebras discussed above,
we introduce the vector space $\mathcal{E}_{\delta}(k|l)$, which has a basis consisting of the basis elements of $\mathcal{E}(k|l)$
and the additional element ${\mathcal E}_0=\delta$. We extend the bilinear form on $\mathcal{E}(k|l)$ to $\mathcal{E}_\delta(k|l)$ by setting
\[
(\mathcal{E}_0, \mathcal{E}_i)=(\mathcal{E}_i, \mathcal{E}_0)=0, \ \forall i=0, 1, \dots, k+l.
\]
The resulting form still has rank $k+l$ and is degenerate.
The affine root data can be described as in Table \ref{table:affine} (see \cite{K1, K2, JWV})  using the ambient space $\cE_\delta(m|n)$ in each case.

\begin{table}[h]
\renewcommand{\arraystretch}{1.2}
\centering
\caption{Classical series of affine Lie superalgebras}
\label{table:affine}
\begin{tabular}{c|c}
\hline
 $\g$ &\text{simple roots}\\
\hline
\multirow{2}{*}{$\Sl(m|n)^{(1)}$} &$\alpha_i={\mathcal E}_i - {\mathcal E}_{i+1}$, \ \ $1\le i< m+n$,\\
& $\alpha_0={\mathcal E}_0-{\mathcal E}_1 + {\mathcal E}_{m+n}.$\\
\hline
\multirow{3}{*}{$\osp(2m+1|2n)^{(1)}$}  &$\alpha_i={\mathcal E}_i - {\mathcal E}_{i+1}$, \ \ $1\le i<m+n$, \ \  $\alpha_{m+n}= {\mathcal E}_{ m+n}$,\\
					& $\alpha_0=\begin{cases}
{\mathcal E}_0-{\mathcal E}_1-{\mathcal E}_{2},&\text{if ${\mathcal E}_1=\varepsilon_{1}$},\\
{\mathcal E}_0-2{{\mathcal E}_1},&\text{if ${\mathcal E}_1=\delta_{1}$}.
\end{cases}$ \\
\hline
\multirow{6}{*}{$\osp(2m|2n)^{(1)}$}  & $\alpha_i={\mathcal E}_i - {\mathcal E}_{i+1},\ \ 1\le i<m+n,$ \\
&$\alpha_{m+n}=\begin{cases}
{\mathcal E}_{m+n-1}+{\mathcal E}_{m+n}, &\text{ if ${\mathcal E_{m+n}}={\varepsilon}_{m}$},\\
2{\mathcal E_{m+n}}, &\text{ if ${\mathcal E_{m+n}}={\delta}_{n}$},
\end{cases}$\\
& $\alpha_0=\begin{cases}
{\mathcal E}_0-{\mathcal E}_1-{\mathcal E}_{2},&\text{if ${\mathcal E}_1=\varepsilon_{1}$},\\
{\mathcal E}_0-2{\mathcal E}_1,&\text{if ${\mathcal E}_1=\delta_{1}$}.
\end{cases}$\\
\hline
\multirow{3}{*}{$\Sl(2m+1|2n)^{(2)}$}  &$\alpha_i={\mathcal E}_i - {\mathcal E}_{i+1}, \ \ 1\le i<m+n ,\ \ \alpha_{m+n}= {\mathcal E}_{ m+n},$\\
 &$\alpha_0=\begin{cases}
{\mathcal E}_0-2{\mathcal E}_1,&\text{if ${\mathcal E}_1=\varepsilon_{1}$},\\
{\mathcal E}_0-{\mathcal E}_1-{\mathcal E}_2,&\text{if ${\mathcal E}_1=\delta_{1}$}.
\end{cases}$ \\
\hline
\multirow{6}{*}{$\Sl(2m|2n)^{(2)}$}  &$\alpha_i={\mathcal E}_i - {\mathcal E}_{i+1} \ \ 1\le i<m+n,$ \\
						 &$\alpha_{m+n}=\begin{cases}
{\mathcal E}_{m+n-1}+{\mathcal E}_{m+n}, &\text{ if ${\mathcal E_{m+n}}={\varepsilon}_{m}$}, \\
2{\mathcal E_{m+n}}, &\text{ if ${\mathcal E_{m+n}}={\delta}_n$},
\end{cases}$\\
					 &$\alpha_0=\begin{cases}
{\mathcal E}_0-2{{\mathcal E}_1},&\text{if ${\mathcal E}_1=\varepsilon_{1}$},\\
{\mathcal E}_0-{\mathcal E}_1-{\mathcal E}_{2},&\text{if ${\mathcal E}_1=\delta_{1}$}.
\end{cases}$\\
\hline
 \multirow{2}{*}{$\osp(2m+2|2n)^{(2)}$} &$\alpha_{i}={\mathcal E}_{i}-{\mathcal E}_{i+i} \ \ 1\le  i< m+n ,$\\ &$\alpha_{m+n}=\cE_{m+n},\ \ \ \alpha_{0}={\mathcal E}_0-{\mathcal E}_{1}.$\\
\hline
  \multirow{2}{*}{$\Sl(2m+1|2n+1)^{(4)}$} &$\alpha_{i}={\mathcal E}_{i}-{\mathcal E}_{i+i} \ \ 1\le  i< m+n ,$\\
&$\alpha_{m+n}=\cE_{m+n},\ \ \ \alpha_{0}={\mathcal E}_0-{\mathcal E}_{1}.$\\
\hline
\end{tabular}
\end{table}

The vector spaces ${\mathcal E}_{\delta}(m|n)$ and ${\mathcal E}_{\delta}(n|m)$ are both $(m+n+1)$-dimensional. To avoid confusion, we write the basis of ${\mathcal E}_{\delta}(n|m)$ as $\{\delta', \ve'_1,\dots \ve'_n,\delta'_1,\dots \delta'_m\}$. Consider the following
vector space isomorphism
\begin{eqnarray}\label{eq:phi-def}
\phi:  {\mathcal E}_{\delta}(m|n)\longrightarrow {\mathcal E}_{\delta}(n|m), \quad \delta\mapsto\delta', \ \
\varepsilon_i\mapsto \delta'_i, \  \  \delta_j\mapsto\varepsilon'_j, \  \  \forall i, j.
\end{eqnarray}
We will still denote its restriction to ${\mathcal E}(m|n)$ by $\phi$.

Clearly $\phi$ sends an admissible basis of ${\mathcal E}(m|n)$ to an admissible basis of
${\mathcal E}(n|m)$.  If ${\mathcal E}(m|n)$ or ${\mathcal E}_\delta(m|n)$ is the ambient space of the root system of a Lie superalgebra or affine Lie superalgebra $\g$, and $\Pi$ is a fundamental system of $\g$, then the set $\phi(\Pi)$ may be a fundamental system of another (affine) Lie superalgebra with ${\mathcal E}(n|m)$ or ${\mathcal E}_{\delta}(n|m)$ as the ambient space of roots. This happens in the following cases.

\begin{lemma}\label{lem:connection phi}
The map $\phi$ induces a one to one correspondences between fundamental systems of the (affine) Lie superalgebras in  each of the following pairs $(\g, \g')$:
\begin{enumerate}
\item[(i).] those listed in Table \ref{tbl:q-correspond};
\item[(ii).] \label{sl-pairs}  and $(\Sl(m|n), \Sl(n|m))$, $(\Sl(m|n)^{(1)}, \Sl(n|m)^{(1)})$, $(\Sl(2m|2n)^{(2)}, \Sl(2n|2m)^{(2)})$,\\  $(\Sl(2m+1|2n)^{(2)}, \Sl(2n|2m+1)^{(2)})$, $(\Sl(2m+1|2n+1)^{(4)},\Sl(2n+1|2m+1)^{(4)})$.
\end{enumerate}
\end{lemma}

\begin{remark}
(1). The imaginary root  ${\mathcal E}_0=\delta$ is even for all affine Lie superalgebras except $\Sl(2m+1|2n+1)^{(4)}$, where it is odd. \\  
(2). Define a map $\chi_{\g}:  \Z\delta+\cE(m|n)_{\Z} \longrightarrow \Z$ for each $\g$ in Table  \ref{table:affine} as follows. For any $v=z_0\delta+v'$ with $v'\in \cE(m|n)_{\Z}$, let $\chi_\g(v)=z_0+\chi(v')$ 
if $\g=\Sl(2m+1|2n+1)^{(4)}$,  and $\chi_\g(v)=\chi(v')$ otherwise. Then a simple root $\alpha_i$ is even if $\chi_\g(\alpha_i)$  is even, and odd otherwise. \\
(3). For all the pairs in case (ii) of Lemma \ref{lem:connection phi}, we have $\g=\g'$. This is why we do not consider them when studying quantum correspondences.
\end{remark}

\subsection{Quantum affine superalgebras} Hereafter we will only
consider the Lie superalgebras in Table \ref{table:classical} and  affine Lie superalgebras in Table \ref{table:affine}.

Let $\g$ be such a Lie superalgebra or affine Lie superalgebra with  a fundamental system $\Pi$.
For $\g$ in Table \ref{table:classical}, let $\Pi=\{\alpha_i\mid i=1, 2, \dots, m+n\}$, and let  $\tau\subset\{1, 2, \dots, m+n\}$ be the labelling set of the odd simple roots, i.e.,  $\{\alpha_s\mid s\in\tau\}$ is the subset of $\Pi$ consisting of the odd simple roots.  Similarly, for $\g$ in Table \ref{table:affine}, let $\Pi=\{\alpha_i\mid i=0, 1, 2, \dots, m+n\}$, and let  $\tau\subset\{0, 1, 2, \dots, m+n\}$ be the labelling set of the odd simple roots.
Define $b_{i j}=(\al_i,\al_j)$ for all $i, j$. Then the
Cartan matrix of $\g$ corresponding to $\Pi$ is given by
\[
A=(a_{ij}) \quad\text{with}\quad
a_{ij}=\begin{cases}\dfrac{2b_{ij}}{b_{ii}},&\mbox{if}~b_{ii}\neq 0\\b_{ij},&\mbox{if}~b_{ii}=0
\end{cases}.
\]
Note that $a_{i i}=0$ if and only if $\alpha_i$ is an isotropic odd simple root.
We will represent fundamental systems by Dynkin diagrams \cite{K1,K2,JWV, Z1},
following the convention of Kac \cite{K1}.  In particular, a node
$\circ$ corresponds to an even simple root;
$\otimes$ to an odd isotropic simple root;
$\bullet$ to an odd non-isotropic simple root, and
$\times$ stands for $\circ$ or $\otimes$, depending on whether the simple root is even or odd.  Note that
the sub-diagrams
\begin{picture}(72, 20)(4, 6)
\put(10, 10){\circle{10}}
\put(16, 10){\line(1, 0){18}}
\put(15,7){$<$}
\put(35, 6){\Large$\otimes$}
\put(46, 10){\line(1, 0){18}}
\put(26, 12){\tiny $2$}
\put(58,7){$>$}
\put(65, 6){\Large$\otimes$}
\end{picture}
 and
 \begin{picture}(75, 20)(25, 6)
\put(30, 10){\circle{10}}
\put(37, 10){\line(1, 0){18}}
\put(35,7){$<$}
\put(55, 6){\Large$\otimes$}
\put(65, 10){\line(1, 0){18}}
\put(47, 12){\tiny $2$}
\put(78,7){$>$}
\put(90, 10){\circle{10}}
\end{picture}
correspond respectively to the sub-matrices  $\begin{bmatrix} 2 &-1 &0\\ -2& 0& 1\\0 &1 &0 \end{bmatrix}$ and $\begin{bmatrix} 2 &-1 &0\\ -2& 0& 1\\0 &-1 &2 \end{bmatrix}$ in Cartan matrices.

For convenience, we take a slight variation of the usual definition \cite{BGZ, Y, ZGB} of quantised universal enveloping superalgebras (see Remark \ref{rem:def-change} below for further comments).
Let us
fix $q\in\mb{C}$ such that $q\ne 0, \pm 1$, and let $q^{1/2}$ be a fixed
square root of $q$.  Denote
\[
\begin{aligned}
q_i=\begin{cases}q^{\frac{(\alpha_i,\alpha_i)}{2}},&\mbox{if }(\al_i,\al_i)\neq 0\\q,&\mbox{if }(\al_i,\al_i)=0,
\end{cases}\qquad
\theta_i=\begin{cases}1,&\mbox{if }|(\al_i,\al_i)|=1,2\\2,&\mbox{if } |(\al_i,\al_i)|=0,4.\end{cases}
\end{aligned}
\]
Note that $q_i^{a_{ij}}=q_j^{a_{j,i}}=q^{(\alpha_i,\alpha_j)}$. In what follows, $[x,y]_v=xy-(-1)^{[x][y]}v yx$.
\begin{definition}\label{defi:quantised}
The {\em quantised universal enveloping superalgebra} $\U_q(\g,\Pi)$ of $\g$ with the fundamental system $\Pi$ is an associative superalgebra over $\mb{C}$ with identity, which is defined by the following presentation:
The generators are
$e_i,f_i,k_i^{\pm1}$, where $e_s,f_s, (s\in\tau),$ are odd and the rest are even, and the relations are given by
\begin{enumerate}
\item\qquad\qquad\quad\quad\quad\quad $k_ik_i^{-1}=k_i^{-1}k_i=1,\quad k_ik_j=k_jk_i,$
\[
\begin{array}{l}
k_ie_jk_i^{-1}=q_i^{a_{ij}}e_j,\quad k_if_jk_i^{-1}=q_i^{-a_{ij}}f_j,\\
e_if_j-(-1)^{[e_i][f_j]}f_je_i=\de_{ij}\dfrac{k_i-k_i^{-1}}{q^{\theta_i}-q^{-\theta_i}};
\end{array}
\]
\item\qquad\qquad if $a_{ss}=0$, \qquad\quad $(e_s)^2=(f_s)^2=0,$
\[
\begin{aligned}
&\text{ if } a_{ii}\neq0,i\neq j, \quad \lb\mbox{Ad}_{e_i}\rb^{1-a_{ij}}(e_j)=\lb\mbox{Ad}_{f_i}\rb^{1-a_{ij}}(f_j)=0,
\end{aligned}\]
where $\mbox{Ad}_{e_i}(x)$ and $\mbox{Ad}_{f_i}(x)$ are defined by
\eqref{eq:Ad};

\item  and higher order Serre relations \cite{Y1} (also see \cite{Z1}) associated with the following subdiagrams of Dynkin diagrams
\end{enumerate}

{\renewcommand\baselinestretch{1.2}\selectfont
(A)\quad \label{Serre:case-1}
\begin{picture}(65, 20)(13, 5)
\put(10, 7){$\times$}
\put(15, 10){\line(1, 0){20}}
 \put(35, 6){\Large$\otimes$ }
\put(45, 10){\line(1, 0){20}}
\put(62, 7){$\times$}
\put(8, -2){\tiny $s-1$}
\put(39, -2){\tiny $s$}
\put(60, -2){\tiny $s+1$}
\end{picture}
with $a_{s-1,s}=-a_{s,s+1}$, the
associated higher order Serre relations are
\[\begin{aligned}
\mbox{Ad}_{e_s}\mbox{Ad}_{e_{s-1}}\mbox{Ad}_{e_s}(e_{s+1})=0,\quad
\mbox{Ad}_{f_s}\mbox{Ad}_{f_{s-1}}\mbox{Ad}_{f_s}(f_{s+1})=0;
\end{aligned}\]

(B)\quad \label{Serre:case-2}
\begin{picture}(70, 20)(13, 5)
\put(10, 7){$\times$}
\put(15, 10){\line(1, 0){20}}
\put(35, 6){\Large$\otimes$}
\put(45, 11){\line(1, 0){17}}
\put(45, 9){\line(1, 0){17}}
\put(57, 6.5){$>$}
\put(70,10){\circle{10}}
\put(8, -2){\tiny $s-1$}
\put(39, -2){\tiny $s$}
\put(63, -2){\tiny $s+1$}
\put(76, 7){,}
\end{picture}
the associated higher order Serre elements are
\[\begin{aligned}
\mbox{Ad}_{e_s}\mbox{Ad}_{e_{s-1}}\mbox{Ad}_{e_s}(e_{s+1})=0,\quad
\mbox{Ad}_{f_s}\mbox{Ad}_{f_{s-1}}\mbox{Ad}_{f_s}(f_{s+1})=0;
\end{aligned}\]

(C) \quad \label{case-3}
\begin{picture}(68, 20)(15, 5)
\put(10, 6.5){$\times$}
\put(15, 10){\line(1, 0){20}}
\put(35, 6){\Large$\otimes$ }
\put(45, 11){\line(1, 0){17}}
\put(45, 9){\line(1, 0){17}}
\put(57, 6.5){$>$}
\put(70, 10){\circle*{10}}
\put(8, -2){\tiny $s-1$}
\put(39, -2){\tiny $s$}
\put(63, -2){\tiny $s+1$}
\put(76, 7){,}
\end{picture}
the associated higher order Serre relations are
\[\begin{aligned}
\mbox{Ad}_{e_s}\mbox{Ad}_{e_{s-1}}\mbox{Ad}_{e_s}(e_{s+1})=0,\quad
\mbox{Ad}_{f_s}\mbox{Ad}_{f_{s-1}}\mbox{Ad}_{f_s}(f_{s+1})=0;
\end{aligned}\]

(D)\quad \label{case-4}
\begin{picture}(72, 20)(10, 5)
\put(10, 10){\circle{10}}
\put(16, 10){\line(1, 0){18}}
\put(15,7){$<$}
\put(35, 6){\Large$\otimes$}
\put(46, 10){\line(1, 0){18}}
\put(26, 12){\tiny $2$}
\put(58,7){$>$}
\put(65, 6){\Large$\otimes$}
\put(3, -2){\tiny $s-1$}
\put(39, -2){\tiny $s$}
\put(63, -2){\tiny $s+1$}
\put(76, 7){,}
\end{picture}
the associated higher order Serre relations are
\[
\begin{aligned}
\left[\ad_{e_{s+1}}(e_{s}), \left[\ad_{e_{s+1}}(e_s) , \ad_{e_s}(e_{s-1})\right]_{v_1}\right]_{v_2}=0,\\
\left[\ad_{f_{s+1}}(e_{s}), \left[\ad_{f_{s+1}}(f_s) , \ad_{f_s}(f_{s-1})\right]_{v_1}\right]_{v_2}=0,
\end{aligned}\]
where $v_1=q^{-(\alpha_s,\alpha_{s+1})}, v_2=q^{(\alpha_s,\alpha_{s+1})}$;

(E)\quad \label{case-5}
\begin{picture}(100, 20)(30, 5)
\put(30, 10){\circle{10}}
\put(37, 10){\line(1, 0){18}}
\put(35,6){$<$}
\put(55, 6.5){\Large$\otimes$}
\put(65, 10){\line(1, 0){18}}
\put(47, 12){\tiny $2$}
\put(78,6.5){$>$}
\put(90, 10){\circle{10}}
\put(95, 10){\line(1, 0){20}}
\put(115, 6.5){$\times$}
\put(112, -2){\tiny $s+2$}
\put(26, -2){\tiny $s-1$}
\put(59, -2){\tiny $s$}
\put(82, -2){\tiny $s+1$}
\put(123, 7){,}
\end{picture}
the associated higher order Serre relations are
\[
\begin{aligned}
&\left[\ad_{e_{s+2}}(\ad_{e_{s+1}}e_s),\left[\ad_{e_{s+1}}e_s, \ad_{e_s}e_{s-1}
\right]_{v_1}\right]=0,\\
&\left[\ad_{f_{s+2}}(\ad_{f_{s+1}}f_s),\left[\ad_{f_{s+1}}f_s, \ad_{f_s}f_{s-1}
\right]_{v_1}\right]=0, \ \text{ $v_1=q^{-(\alpha_{s},\alpha_{s+1})}$};
\end{aligned}
\]

(F)\quad \label{case-6}
\begin{picture}(65, 30)(0, 3)
\put(9, 6.8){$\times$}
\put(15, 10){\line(1, 1){20}}
\put(15,10){\line(1, -1){20}}
\put(35, 26){\Large$\otimes$}
\put(39, 25){\line(0, -1){30}}
\put(41, 25){\line(0, -1){30}}
\put(35, -14){\Large$\otimes$}
\put(6, -2){\tiny $s-1$}
\put(47, 28){\tiny $s$}
\put(47, -13){\tiny $s+1$}
\put(50, 5){,}
\end{picture}
the associated higher order Serre relations are\\
\[\begin{aligned}
&\ad_{e_s}\ad_{e_{s+1}}(e_{s-1})-\ad_{e_{s+1}}\ad_{e_s}(e_{s-1})=0,\\
&\ad_{f_s}\ad_{f_{s+1}}(f_{s-1})-\ad_{f_{s+1}}\ad_{f_s}(f_{s-1})=0;
\end{aligned}\]}
where $\mbox{Ad}_{e_i}(x)$ and $\mbox{Ad}_{f_i}(x)$ are defined by
\begin{eqnarray}\label{eq:Ad}
\mbox{Ad}_{e_i}(x)=e_ix-(-1)^{[e_i][x]}k_ixk_i^{-1}e_i, \quad
\mbox{Ad}_{f_i}(x)=f_ix-(-1)^{[f_i][x]}k_i^{-1}xk_if_i.
\end{eqnarray}
\end{definition}

\begin{remark}\label{rem:def-change}  We have used $\dfrac{k_i-k_i^{-1}}{q^{\theta_i}-q^{-\theta_i}}$ instead of the standard expression $\dfrac{k_i-k_i^{-1}}{q_i-q_i^{-1}}$
in the third relation of (1).
A consequence is that $q^{\pm 1/2}$ never  appears in our definition of the quantised universal enveloping superalgebras.
\end{remark}

Corresponding to each simple root $\alpha_i\in \Pi$ of $\g$,
we introduce a  group $\Z_2$ generated by $\sigma_i$ such that $\sigma_i^2=1$, and let $\mathrm{G}$ be the direct product of all such groups.
Then $G=\Z_2^{\times |\Pi|}$, where $|\Pi|$ denotes the cardinality of $\Pi$.
We define a $\mathrm{G}$-action on $\U_q(\g,\Pi)$ by
\[
\begin{aligned}
&\sigma_i\cdot e_j=(-1)^{(\alpha_i,\alpha_j)}e_j, \quad \sigma_i\cdot f_j=(-1)^{-(\alpha_i,\alpha_j)}f_j, \quad \sigma_i\cdot k_j=k_j, \quad\text{$i\ne 0$}, \\
&\sigma_0\cdot e_j=(-1)^{\delta_{i,0}}e_j,\qquad \sigma_0\cdot f_j=(-1)^{-\delta_{i,0}}f_j, \qquad \sigma_0\cdot k_j=k_j, \quad \forall j,
\end{aligned}
\]
where the second line is present only when $\g$ is an affine superalgebra.

\begin{definition}\label{defi:smash product}
Let $\UU_q(\g,\Pi)=\U_q(\g,\Pi)\sharp\C \mathrm{G}$
denote the smash product superalgebra with $\UU_q(\g,\Pi)_{\bar 0}=\U_q(\g,\Pi)_{\bar 0}\otimes\C \mathrm{G}$ and $\UU_q(\g,\Pi)_{\bar 1}=\U_q(\g,\Pi)_{\bar 1}\otimes\C \mathrm{G}$, where the multiplication is defined by
\[
\begin{aligned}
(x\otimes \sigma)(x'\otimes \sigma')=x(\sigma\cdot x')\otimes \sigma\sigma', \quad \forall x, x'\in \U_q(\g,\Pi), \ \sigma,\sigma'\in \C\mathrm{G}.
\end{aligned}
\]
\end{definition}
To simplify the notation, we write $x$ for  $x\otimes 1$ for any $x\in\U_q(\g,\Pi)$, and
$\sigma$ for $1\otimes \sigma$ for any $\sigma\in\C\mathrm{G}$. Then in $\UU_q(\g, \Pi)$, we have
\begin{eqnarray}\label{eq:sigma-act}
\begin{aligned}
&\sigma_i e_j\sigma_i^{-1}=(-1)^{(\alpha_i,\alpha_j)}e_j,\ \
\sigma_i f_j\sigma_i^{-1}=(-1)^{-(\alpha_i,\alpha_j)}f_j, \ \
\sigma_i k_j\sigma_i^{-1}=k_j, \ \  i\ne 0,\\
&\sigma_0 e_j \sigma^{-1}_0=(-1)^{\delta_{i,0}}e_j,\ \quad \sigma_0 f_j \sigma^{-1}_0=(-1)^{-\delta_{i,0}}f_j,\ \ \quad \sigma_0 k_j \sigma^{-1}_0=k_j,  \ \  \forall j.
\end{aligned}
\end{eqnarray}

The quantised universal enveloping algebra $\U_q(\g, \Pi)$ is a Hopf superalgebra, and the group algebra of $G$ has a canonical Hopf algebra structure. Their smash product inherits a Hopf superalgebra structure.
\begin{proposition}\label{prop:hopf}
The quantum superalgebra $\UU_q(\g,\Pi)$ is a Hopf superalgebra with

\noindent
comultiplication $\De:\UU_q(\g,\Pi)\rightarrow \UU_q(\g,\Pi)\otimes \UU_q(\g,\Pi)$,
\[
\begin{aligned}
\De(e_i)=e_i\otimes 1+k_i\otimes e_i, &\quad
\De(f_i)=f_i\otimes k_i^{-1}+1\otimes f_i,\\
\De(k_i)=k_i\otimes k_i, &\quad
\De(\sigma_i)=\sigma_i\otimes\sigma_i;
\end{aligned}
\]
counit $\ve:\UU_q(\g,\Pi)\rightarrow \mb{C}$,\ \
$
\ve(e_i)=\ve(f_i)=0, \ \ \ve(k_i^{\pm1})=1, \ \ \ve(\sigma_i)=1;
$
and
 antipode $S:\UU_q(\g,\Pi)\rightarrow \UU_q(\g,\Pi)$, \ \
$
\begin{aligned}
S(e_i)=-k_i^{-1}e_i,\ \ S(f_i)=-f_ik_i, \ \ S(k_i)=k_i^{-1}, \ \ S(\sigma_i)=\sigma_i^{-1}.
\end{aligned}
$
\end{proposition}

\section{Algebraic isomorphisms}\label{sect:quantum}

We prove Theorem \ref{thm:main-quan} in this section.  The proof requires
detailed considerations of the structures of the relevant quantum superalgebras,
thus is very lengthy as each pair $(\g, \g')$ involves
numerous cases corresponding to different choices of fundamental systems.
We will present only the main steps of the proof, omitting most of the detailed calculations.

Let $(\g, \g')$ be  a pair of  Lie superalgebras or affine Lie superalgebras
in Theorem \ref{thm:main-quan}. Choose any fundamental system $\Pi$  for $\g$ with $\tau$ being the labelling set for the odd simple roots. By Lemma \ref{lem:connection phi},
$\Pi'=\phi(\Pi)$ is the corresponding  fundamental system of $\g'$ with the labelling set $\tau'$ for the odd simple roots.
We write $\alpha'_i=\phi(\alpha_i)$ for the simple roots of $\g'$.   Note that $\alpha'_s\in\Pi'$ is isotropic if and only if $\alpha_s\in\Pi$ is.

Let  $\{e_i, f_i, k_i^{\pm 1}, \sigma_i\}$ be the set of generators of
the quantum superalgebra $\UU_q(\g, \Pi)$, and denote by  $\mathrm{G}$
the group generated by the elements $\sigma_i$. Similarly, we let
$\{e'_i, f'_i, {k'}_i^{\pm 1}, \sigma'_i\}$ be the standard generating set of $\UU_{-q}(\g', \Pi')$, and denote by $\mathrm{G}'$
the group generated by the elements $\sigma'_i$.

For $1\leq i\leq m+n$, we introduce the following elements
\begin{align*}
&\Phi_i=\prod_{k=i}^{m+n}\sigma_k,
\quad \tilde{\Phi}_i=\prod_{k=0}^{m+n}\sigma_{i+2k},
\ \quad \text{in $\UU_q(\g, \Pi)$};\\
&\Phi'_i=\prod_{k=i}^{m+n}\sigma'_k,\quad \tilde{\Phi}'_i=\prod_{k=0}^{m+n}\sigma'_{i+2k},
\quad \text{in $\UU_{-q}(\g', \Pi')$},
\end{align*}
where $\sigma_j\in G$ and $\sigma'_j\in G'$ are both $1$ if $j\geq m+n+1$.  Note that
\[
\Phi_{m+n}=\tilde{\Phi}_{m+n}=\sigma_{m+n}, \quad  \Phi'_{m+n}=\tilde{\Phi}'_{m+n}=\sigma'_{m+n}.
\]

For $i=1,2,\dots,m+n$, we define the following elements
\begin{eqnarray}\label{eq:connect B}
\begin{aligned}
&E_{i}=\Phi_{i+1}e_i,\quad F_{i}=\Phi_if_i,\quad i\notin \tau, \\
&E_i=\tilde{\Phi}_{i+2}e_i,\quad F_i=\tilde{\Phi}_{i}f_i,\quad i\in \tau,\\
&K_i=\sigma_i k_i, \quad \text{which belong to $\UU_q(\g, \Pi)$; and}  \\
&E'_{i}=\Phi'_{i+1}e'_i,\quad F'_{i}=\Phi'_i f'_i,\quad i\notin \tau', \\
&E'_i=\tilde{\Phi}'_{i+2}e'_i,\quad F'_i=\tilde{\Phi}'_{i}f'_i,\quad i\in \tau',\\
&K'_i=\sigma'_i k'_i,  \quad  \text{which belong to $\UU_{-q}(\g', \Pi')$},
\end{aligned}
\end{eqnarray}
where $\Phi_{m+n+k}=\tilde{\Phi}_{m+n+k}=1$ and
$\Phi'_{m+n+k}=\tilde{\Phi}'_{m+n+k}=1$
for all $k>0$.

If $(\g, \g')$ is a pair of affine Lie superalgebras, we will also  define  elements
\[
E_0, F_0, K_0\in \UU_q(\g, \Pi), \quad \text{and} \quad E'_0, F'_0, K'_0\in\UU_{-q}(\g', \Pi'),
\]
the explicit expressions of which depend on the affine Lie superalgebras
and fundamental systems, and will be given in the proof of the following result.

\begin{theorem}\label{thm:iso-main} The associative algebra isomorphism
$ \UU_{-q}(\g', \Pi')\stackrel{\cong}{\longrightarrow}\UU_q(\g, \Pi)$ of Theorem \ref{thm:main-quan} is given by
\begin{equation}\label{eq:B-map}
\sigma'_i\mapsto \sigma_i, \quad e'_i \mapsto E_i, \quad f'_i \mapsto F_i, \quad k'_i \mapsto K_i, \quad \forall i,
\end{equation}
with the inverse map
\begin{equation}\label{eq:B-map-inv}
\sigma_i\mapsto \sigma'_i, \quad e_i \mapsto E'_i, \quad f_i \mapsto F'_i, \quad k_i \mapsto K'_i, \quad \forall i.
\end{equation}
\end{theorem}
This is a more explicit version of Theorem \ref{thm:main-quan}.
If we can prove that the maps \eqref{eq:B-map} and \eqref{eq:B-map-inv} are algebra homomorphisms,
then we immediately see that they are inverses of each other since $\Phi^2_i=1$ for all $i$.
It is clear from equation \eqref{eq:sigma-act} that the maps preserve the action of $G$ on $\Uq(\g, \Pi)$ and the action of $G'$ on $\U_{-q}(\g', \Pi')$.  Thus what remains to be shown is that
\begin{itemize}
\item $E_i, F_i, K^{\pm 1}_i$ satisfy the defining relations of $\U_{-q}(\g', \Pi')$
obeyed by the standard generators $e'_i, f'_i, {k'}_i^{\pm 1}$ ($1\le i\le m+n$);
and
\item
$E'_i, F'_i, {K'}^{\pm 1}_i$ satisfy the defining relations of $\U_q(\g, \Pi)$ obeyed by the standard generators $e_i, f_i, k_i^{\pm 1}$ ($1\le i\le m+n$).
\end{itemize}
The proof simply boils down to deducing the desired relations satisfied
by $E_i, F_i, K^{\pm 1}_i$ (rsep.  $E'_i, F'_i, {K'}^{\pm 1}_i$), from the defining relations of  $\U_q(\g, \Pi)$ (resp. $\U_{-q}(\g', \Pi')$).
The proofs for the two statements are exactly the same, thus
we will only present the details for the first one, which will occupy the next three sections. 
The following notation will be used,
\[\begin{aligned}
&[k]_z=\frac{z^k-z^{-k}}{z-z^{-1}},\quad \{k\}_z=\frac{z^k-(-z)^{-k}}{z+z^{-1}},\quad \mbox{for}\ \  k\in\Z, \\
&[0]_z!=\{0\}_z!=1, \quad [N]_z!=\prod_{i=1}^N[i]_z,\quad \{N\}_z!=\prod_{i=1}^N\{i\}_z, \mbox{ for}\ \ 1\le N\in\N,\\
&\begin{bmatrix} N\\k\end{bmatrix}_z=\frac{[N]_z!}{[N-k]_z![k]_z!},\quad \left\{\begin{matrix} N\\k\end{matrix}\right\}_z=\frac{\{N\}_z!}{\{N-k\}_z! \{k\}_z!},\quad \mbox{for}\ \  k\leq N \in \N,
\end{aligned}
\]
where $z\in\C$ such that the expressions above are defined.

\subsection{The case of  $\osp(2m+1|2n)$ and $\osp(2n+1|2m)$ }\label{sect:osp}

Recall from Section \ref{sect:roots} that
the ambient space of the roots of $\g=\osp(2m+1|2n)$ is $\cE(m|n)$.
Each admissible ordered basis of it leads to a fundamental system $\Pi=\{\alpha_i\mid 1\le i\le m+n\}$ with the $\alpha_i$ given in
Table \ref{table:classical}.
Now $\g'=\osp(2n+1|2m)$ with the corresponding fundamental system $\Pi'=\{\alpha'_i=\phi(\alpha_i)\mid i=1, 2, \dots, m+n\}$.
The ambient space of the roots is $\cE(n|m)$. In the case $\alpha_{m+n}=\delta_{n}\in \Pi$, which is odd, $\alpha'_{m+n}=\varepsilon_n$ is an even simple root in $\Pi'$, and the
Dynkin diagrams of $\Pi$ and $\Pi'$ are the Type (1) diagrams in Table \ref{table:Dynkin diagram-B}. In this case,  $\tau'=\tau\backslash\{m+n\}$.   If $\alpha_{m+n}=\varepsilon_m\in\Pi$,
which is even,  $\alpha'_{m+n}=\delta_m$ is an odd simple root in $\Pi'$, and the
Dynkin diagrams of $\Pi$ and $\Pi'$ are the Type (2) diagrams in Table \ref{table:Dynkin diagram-B}.  In this case, $\tau=\tau'\backslash\{m+n\}$.

\begin{table}[h]
\caption{Dynkin diagram of $\osp(2m+1|2n)$ and $\osp(2n+1|2m)$}
\label{table:Dynkin diagram-B}
\begin{tabular}{ >{\centering\arraybackslash}m{0.4in} | >{\centering\arraybackslash}m{2.2in} | >{\centering\arraybackslash}m{2.2in}  }
\hline
 Type
&\vspace{2mm} $\g=\osp(2m+1|2n)$ \vspace{2mm} & $\g'=\osp(2n+1|2m)$\\
\hline
(1) &
\begin{picture}(0, 25)(75, 2)
\put(10, 10.5){$\times$}
\put(10,2){\tiny$\alpha_1$}
\put(15, 14){\line(1, 0){20}}
 \put(35, 10.5){$\times$}
\put(35,2){\tiny$\alpha_2$}
\put(42, 14){\line(1, 0){20}}
\put(65,11){$\cdots$}
\put(80, 14){\line(1, 0){20}}
\put(100, 11){$\times$}
\put(90,2){\tiny$\alpha_{m+n-1}$}
\put(106, 15){\line(1, 0){17}}
\put(106, 13){\line(1, 0){17}}
\put(118, 10.5){$>$}
\put(130, 14){\circle*{10}}
\put(123,2){\tiny$\alpha_{m+n}$}
\end{picture}
&
\begin{picture}(0, 25)(75, 2)
\put(10, 10.5){$\times$}
\put(10,2){\tiny$\alpha'_1$}
\put(15, 14){\line(1, 0){20}}
 \put(35, 10.5){$\times$}
\put(35,2){\tiny$\alpha'_2$}
\put(42, 14){\line(1, 0){20}}
\put(65,11){$\cdots$}
\put(80, 14){\line(1, 0){20}}
\put(100, 11){$\times$}
\put(90,2){\tiny$\alpha'_{m+n-1}$}
\put(106, 15){\line(1, 0){17}}
\put(106, 13){\line(1, 0){17}}
\put(118, 10.5){$>$}
\put(130, 14){\circle{10}}
\put(123,2){\tiny$\alpha'_{m+n}$}
\end{picture}\\
\hline
(2)
&\begin{picture}(0, 25)(75, 2)
\put(10, 10.5){$\times$}
\put(10,2){\tiny$\alpha_1$}
\put(15, 14){\line(1, 0){20}}
 \put(35, 10.5){$\times$}
\put(35,2){\tiny$\alpha_2$}
\put(42, 14){\line(1, 0){20}}
\put(65,11){$\cdots$}
\put(80, 14){\line(1, 0){20}}
\put(100, 11){$\times$}
\put(90,2){\tiny$\alpha_{m+n-1}$}
\put(106, 15){\line(1, 0){17}}
\put(106, 13){\line(1, 0){17}}
\put(118, 10.5){$>$}
\put(130, 14){\circle{10}}
\put(123,2){\tiny$\alpha_{m+n}$}
\end{picture}
&
\begin{picture}(0, 25)(75, 2)
\put(10, 10.5){$\times$}
\put(10,2){\tiny$\alpha'_1$}
\put(15, 14){\line(1, 0){20}}
 \put(35, 10.5){$\times$}
\put(35,2){\tiny$\alpha'_2$}
\put(42, 14){\line(1, 0){20}}
\put(65,11){$\cdots$}
\put(80, 14){\line(1, 0){20}}
\put(100, 11){$\times$}
\put(90,2){\tiny$\alpha'_{m+n-1}$}
\put(106, 15){\line(1, 0){17}}
\put(106, 13){\line(1, 0){17}}
\put(118, 10.5){$>$}
\put(130, 14){\circle*{10}}
\put(123,2){\tiny$\alpha'_{m+n}$}
\end{picture}\\
\hline
\end{tabular}
\end{table}


The quantum superalgebra $\UU_q(\osp(2m+1|2n), \Pi)$ is generated by
$\Uq(\osp(2m+1|2n), \Pi)$, and the elements $\sigma_i$, which generate a group $G=\Z_2^{\times (m+n)}$.
The commutation relations of the $\sigma_i$ with the generators of $\Uq(\osp(2m+1|2n), \Pi)$ are given by \eqref{eq:sigma-act}.
The quantum superalgebra $\UU_{-q}(\osp(2n|2m+1), \Pi')$ can be described similarly.

Let $t=-q$, and let $t^{1/2}$ be a square root  for $t$. Denote $t_i=t^{(\alpha_i,\alpha_i)/2}$. Note that relations among $E_i, F_i, K_i$ depend on $t^{\pm 1}$, but not  on $t^{\pm 1/2}$.

\begin{proof}[Proof of Theorem \ref{thm:iso-main} for $(\g, \g')=(\osp(2m+1|2n), \osp(2n+1|2m))$]
Consider first the Type (1) Dynkin diagrams in Table \ref{table:Dynkin diagram-B}.
In this case, $\alpha_{m+n}=\delta_{n}\in \Pi$. Let $\bar{i}=1$ if $i\in\tau'$ and $0$ otherwise. Then it is easy to see that
\begin{equation}\label{eq:EF-t}
E_i F_j - (-1)^{\bar{i}\bar{j}}F_j E_i
    =\delta_{i j} \frac{K_i -K_i^{-1}}{t ^{\theta_i}- t^{-\theta_i}}, \quad 1\le i, j\le m+n.
\end{equation}
For example,
\[\begin{aligned}
&E_{m+n}F_{m+n}\!\!-\!\!F_{m+n}E_{m+n}=-\sigma_{m+n}(e_{m+n}f_{m+n}\!+\!f_{m+n}e_{m+n})=\!\!\frac{K_{m+n}\!-\!K_{m+n}^{-1}}{t\!-\!t^{-1}}, \\
&E_kF_{m+n}-F_{m+n}E_k=(-1)^{\delta_{k,m+n-1}}\tilde{\Phi}_{k+2}\tilde{\Phi}_{m+n}(e_kf_{m+n}+f_{m+n}e_k)=0, \ \ k\in\tau'.
\end{aligned}\]
The other relations listed in $(1)$ of Definition \ref {defi:quantised} can be proved in exactly the same way.

We now consider the Serre relations $(2)$ in Definition \ref {defi:quantised}, which lead to
\begin{eqnarray}\label{eq:serre-t}
\begin{aligned}
&E_i^2=F_i^2=0, \quad \text{if $\alpha_i\in\Pi$ is isotropic},\\
&\sum_{k=0}^{1-a_{ij}}(-1)^k{\begin{bmatrix}\begin{smallmatrix} 1-a_{ij}\\k\end{smallmatrix}\end{bmatrix}}_{t_i}E_i^kE_jE_i^{1-a_{ij}-k}=0,\\
&\sum_{k=0}^{1-a_{ij}}(-1)^k{\begin{bmatrix}\begin{smallmatrix} 1-a_{ij}\\k\end{smallmatrix}\end{bmatrix}}_{t_i}F_i^kF_jF_i^{1-a_{ij}-k}=0, \quad  i\not\in\tau', \  i\ne j.
\end{aligned}
\end{eqnarray}
To see this, we consider, for example, the case $i=m+n$ and $j=m+n-1\not\in\tau$.
We have
\[
\begin{aligned}
\sum_{k=0}^3(-1)^k\begin{bmatrix} 3\\k\end{bmatrix}_{t_{m+n}}E_{m+n}^kE_{m+n-1} E_{m+n}^{3-k}
=\Phi_{m+n} \sum_{k=0}^3(-1)^{\frac{k(k+1)}{2}}\left\{\begin{matrix} 3\\k\end{matrix}\right\}_{q_{m+n}} e_{m+n}^k e_{m+n-1} e_{m+n}^{3-k}=0,
\end{aligned}
\]
where we have used
 $
\begin{bmatrix} 3\\k\end{bmatrix}_{t_{m+n}}=(-1)^{\frac{k(k+1)}{2}}\left\{\begin{matrix}3\\k\end{matrix}\right\}_{q_{m+n}}
$
for $0\leq k\leq 3$. For $m+n-1\in\tau$, the above equation remains valid if we replace the expression between the equality signs by $\sum\limits_{k=0}^3(-1)^{\frac{k(k-1)}{2}}\left\{\begin{matrix} 3\\k\end{matrix}\right\}_{q_{m+n}} e_{m+n}^k e_{m+n-1} e_{m+n}^{3-k}$.

Higher order Serre relations can arise from two types of sub-diagrams only in the present case. Denote
\begin{eqnarray}\label{eq:Uab-higher order ef}
\begin{aligned}
e_{i;s;j}=&e_i(e_se_j-(-1)^{[e_j]}q_j^{a_{js}}e_je_s)\\
		&-(-1)^{[e_i](1+[e_j])}q_i^{a_{is}+a_{ij}}(e_se_j-(-1)^{[e_j]}q_j^{a_{js}}e_je_s)e_i,\\
f_{i;s;j}=&f_i(f_sf_j-(-1)^{[f_j]}q_j^{a_{js}}f_jf_s)\\
	&-(-1)^{[f_i](1+[f_j])}q_i^{a_{is}+a_{ij}}(f_sf_j-(-1)^{ [f_j]}q_j^{a_{js}}f_jf_s)f_i.
\end{aligned}
\end{eqnarray}
\begin{case}
\begin{picture}(75, 15)(0, 7)
\put(10, 7){$\times$}
\put(15, 10.5){\line(1, 0){20}}
 \put(35, 6){\Large$\otimes$ }
\put(45, 10.5){\line(1, 0){20}}
\put(62, 7){$\times$}
\put(70, 7){,}
\put(7,0){\tiny $s-1$}
\put(38,-1){\tiny $s$}
\put(59,0){\tiny $s+1$}
\end{picture}
with the associated higher order Serre relations given by
\begin{eqnarray}\label{eq:UqB-higher order1}
\begin{aligned}
&e_s e_{s-1; s; s+1}+(-1)^{[e_{s-1}]+[e_{s+1}]}  e_{s-1; s; s+1}e_s=0,\\
&f_s f_{s-1; s; s+1} +(-1)^{[f_{s-1}]+[f_{s+1}]} f_{s-1; s; s+1}f_s=0.
\end{aligned}
\end{eqnarray}
\end{case}
\noindent

Let us first assume that $s-1,s+1\notin\tau$.
In this case, $s\neq m+n-1$, $a_{s-1,s}=a_{s+1,s}=-1$  and $q_{s+1}=q_{s-1}^{-1}$,
with $q_{s-1}=q$ or $q^{-1}$ depending on depending on
the value of  $\theta$ in \eqref{eq:bilinear form}.  Then \eqref{eq:UqB-higher order1} is given by
\begin{eqnarray*}
\begin{aligned}
e_s e_{s-1; s; s+1} +  e_{s-1; s; s+1}e_s=0,\quad
f_s f_{s-1; s; s+1} +  f_{s-1; s; s+1}f_s=0, \quad \text{with}\\
e_{s-1; s; s+1} = e_{s-1}(e_s e_{s+1}-q_{s+1}^{-1} e_{s+1} e_s)-
q_{s-1}^{-1}(e_s e_{s+1}-q_{s+1}^{-1} e_{s+1} e_s) e_{s-1},\\
f_{s-1; s; s+1} = f_{s-1}(f_s f_{s+1}-q_{s+1}^{-1} f_{s+1} f_s)-
q_{s-1}^{-1}(f_s f_{s+1}-q_{s+1}^{-1} f_{s+1} f_s) f_{s-1}.
\end{aligned}
\end{eqnarray*}
Write $t_{s\pm 1}= - q_{s\pm 1}$, and let
\[\begin{aligned}
E_{s-1;s;s+1}&:=E_{s-1}(E_sE_{s+1}-t_{s+1}^{-1}E_{s+1}E_s)-t_{s-1}^{-1}(E_sE_{s+1}- t_{s+1}^{-1} E_{s+1}E_s)E_{s-1}, \\
F_{s-1;s;s+1}&:=F_{s-1}(F_s F_{s+1}-t_{s+1}^{-1}F_{s+1}F_s)-t_{s-1}^{-1}(F_s F_{s+1}- t_{s+1}^{-1} F_{s+1}F_s)F_{s-1}.
\end{aligned}
\]
For any mutually distinct $i,j,k$ not in $\tau$,
\[
\begin{aligned}
E_iE_jE_k=(-1)^{\delta_{i,k+1}+\delta_{j,k+1}
+\delta_{i,j+1}}\Phi_{i+1}\Phi_{j+1}\Phi_{k+1}e_ie_je_k. 
\end{aligned}\]
If any one of $i, j, k$ is in $\tau$, say,  $j\in\tau$, the identity still holds if we replace $\Phi_{j+1}$ by $\tilde{\Phi}_{j+2}$.
There are also similar relations for $F$'s.  Using these facts, we obtain
\[\begin{aligned}
E_{s-1;s;s+1}&=\Phi_s\tilde{\Phi}_{s+2}\Phi_{s+2}e_{s-1;s;s+1},
\quad
F_{s-1;s;s+1}=\Phi_{s-1}\tilde{\Phi}_{s}\Phi_{s+1}f_{s-1;s;s+1}.
\end{aligned}\]
This immediately leads to
\[
\begin{aligned}
&E_sE_{s-1;s;s+1}+E_{s-1;s;s+1}E_s=-\sigma_s\sigma_{s+1}(e_se_{s-1;s;s+1}+e_{s-1;s;s+1}e_s)=0,\\
&F_sF_{s-1;s;s+1}+F_{s-1;s;s+1}F_s=-\sigma_{s-1}\sigma_{s}(f_sf_{s-1;s;s+1}+f_{s-1;s;s+1}f_s)=0.
\end{aligned}
\]
Without assuming $s-1,s+1\notin\tau$, we can still show that similar relations hold.

In summary,  for $s$ such that $\alpha_s$ is isotropic, we have
\begin{eqnarray}\label{eq:Hserre1-t}
\begin{aligned}
E_sE_{s-1;s;s+1} - (-1)^{1+\overline{s-1}+\overline{s+1}}E_{s-1;s;s+1}E_s=0, \\
F_sF_{s-1;s;s+1} - (-1)^{1+\overline{s-1}+\overline{s+1}}F_{s-1;s;s+1}F_s=0,
\end{aligned}
\end{eqnarray}
where  $\overline{s-1}$ and $\overline{s+1}$ are as in equation \eqref{eq:EF-t}.

\begin{case}
\begin{picture}(82, 15)(0, 7)
\put(10, 7){$\times$}
\put(15, 10){\line(1, 0){20}}
\put(35, 6){\Large$\otimes$}
\put(45, 11){\line(1, 0){17}}
\put(45, 9){\line(1, 0){17}}
\put(57, 6.5){$>$}
\put(70,10){\circle*{10}}
\put(79, 6){,}
\put(7,0){\tiny $s-1$}
\put(38,-1){\tiny $s$}
\put(62,-1){\tiny $s+1$}
\end{picture}
with the associated higher order Serre relations given by
\begin{eqnarray}\label{eq:UqB-higher order3}
\begin{aligned}
&e_s e_{s-1; s; s+1}-(-1)^{[e_{s-1}]}  e_{s-1; s; s+1}e_s=0,\\
&f_s f_{s-1; s; s+1} -(-1)^{[f_{s-1}]}  f_{s-1; s; s+1}f_s=0;
\end{aligned}
\end{eqnarray}
\end{case}
\noindent
where $s+1=m+n$. By the similar method above, we have
\begin{eqnarray}\label{eq:Hserre2-t}
\begin{aligned}
E_sE_{s-1;s;s+1} - (-1)^{1+\overline{s-1}}E_{s-1;s;s+1}E_s=0, \\
F_sF_{s-1;s;s+1} - (-1)^{1+\overline{s-1}}F_{s-1;s;s+1}F_s=0.
\end{aligned}
\end{eqnarray}

Note that equations, \eqref{eq:EF-t}, \eqref{eq:serre-t}, \eqref{eq:Hserre1-t}  and \eqref{eq:Hserre2-t} are
the same as the defining relations of $\U_{-q}(\g', \Pi')$ satisfied by
the generators $e'_i, f'_i, k'_i$. Thus we have shown that the map
$\UU_{-q}(\g', \Pi')\longrightarrow \UU_q(\g, \Pi)$ given by \eqref{eq:B-map}
is indeed an algebra homomorphism if $\alpha_{m+n}=\delta_n$,
i.e., in the case of the Type (1) diagrams in Table \ref{table:Dynkin diagram-B}.
Similarly we can prove this for
Type (2) diagrams in Table \ref{table:Dynkin diagram-B}, where $\alpha_{m+n}=\varepsilon_m$.
\end{proof}

\subsection{The case of  $\Sl(2m+1|2n)^{(2)}$ and $\osp(2n+1|2m)^{(1)}$ }\label{sect:sl-osp}
\setcounter{case}{0}
Given a fundamental system $\Pi=\{\alpha_i\mid i=0, 1, \dots, m+n\}$ of  $\g=\Sl(2m+1|2n)^{(2)}$, we obtain a corresponding fundamental system $\Pi'=\{\alpha'_i=\phi(\alpha_i)\mid i=0, 1, \dots, m+n\}$ of  $\g'=\osp(2n+1|2m)^{(1)}$ by Lemma \ref{lem:connection phi}. We draw the Dynkin diagrams for $\Pi$ and $\Pi'$ in a row of Table \ref{table:Dynkin diagram-sl2}, with the diagram for $\Pi$ on the left. The Dynkin diagrams corresponding to different choices of fundamental systems are divided into four types in the table.
\begin{table}[h]
\caption{Dynkin diagrams of $\Sl(2m+1|2n)^{(2)}$ and $\osp(2n+1|2m)^{(1)}$}
\label{table:Dynkin diagram-sl2}
\begin{tabular}{ >{\centering\arraybackslash}m{0.4in} | >{\centering\arraybackslash}m{2.5in}|  >{\centering\arraybackslash}m{2.5in}  }
\hline
Type  &\vspace{3mm} $\g=\Sl(2m+1|2n)^{(2)}$ \vspace{3mm} & $\g'=\osp(2n+1|2m)^{(1)}$\\
\hline
\multirow{3}{*}{(1)}
&\begin{picture}(180, 28)(-5,-12)
\put(7,0){\circle{10}}
\put(5,-12){\tiny\mbox{$\alpha_0$}}
\put(13,1){\line(1, 0){17}}
\put(13,-1){\line(1, 0){17}}
\put(28,-3){$>$}
\put(40, 0){\circle{10}}
\put(38,-12){\tiny\mbox{$\alpha_1$}}
\put(45, 0){\line(1, 0){20}}
\put(64,-3){$\times$}
\put(70, 0){\line(1, 0){20}}
\put(91, -0.5){\dots}
\put(105, 0){\line(1, 0){20}}
\put(124, -3){$\times$}
\put(130,1){\line(1, 0){17}}
\put(130,-1){\line(1, 0){17}}
\put(143,-3){$>$}
\put(155, 0){\circle*{10}}
\put(150,-12){\mbox{\tiny$\alpha_{m+n}$}}
\end{picture}
& \begin{picture}(180, 28)(-5,-12)
\put(7,0){\circle{10}}
\put(5,-12){\tiny\mbox{$\alpha'_0$}}
\put(13,1){\line(1, 0){17}}
\put(13,-1){\line(1, 0){17}}
\put(28,-3){$>$}
\put(40, 0){\circle{10}}
\put(38,-12){\tiny\mbox{$\alpha'_1$}}
\put(45, 0){\line(1, 0){20}}
\put(64,-3){$\times$}
\put(70, 0){\line(1, 0){20}}
\put(91, -0.5){\dots}
\put(105, 0){\line(1, 0){20}}
\put(124, -3){$\times$}
\put(130,1){\line(1, 0){17}}
\put(130,-1){\line(1, 0){17}}
\put(143,-3){$>$}
\put(155, 0){\circle{10}}
\put(150,-12){\mbox{\tiny$\alpha'_{m+n}$}}
\end{picture}  \\

\cline{2-3}
&\begin{picture}(180, 28)(-5,-12)
\put(7,0){\circle{10}}
\put(5,-12){\tiny$\alpha_0$}
\put(13,1){\line(1, 0){17}}
\put(13,-1){\line(1, 0){17}}
\put(28,-3){$>$}
\put(40, 0){\circle{10}}
\put(38,-12){\tiny$\alpha_1$}
\put(45, 0){\line(1, 0){20}}
\put(64,-3){$\times$}
\put(70, 0){\line(1, 0){20}}
\put(91, -0.5){\dots}
\put(105, 0){\line(1, 0){20}}
\put(124, -3){$\times$}
\put(130,1){\line(1, 0){17}}
\put(130,-1){\line(1, 0){17}}
\put(143,-3){$>$}
\put(155, 0){\circle{10}}
\put(150,-12){\tiny$\alpha_{m+n}$}
\end{picture}
&\begin{picture}(180, 28)(-5,-12)
\put(7,0){\circle{10}}
\put(5,-12){\tiny\mbox{$\alpha'_0$}}
\put(13,1){\line(1, 0){17}}
\put(13,-1){\line(1, 0){17}}
\put(28,-3){$>$}
\put(40, 0){\circle{10}}
\put(38,-12){\tiny\mbox{$\alpha'_1$}}
\put(45, 0){\line(1, 0){20}}
\put(64,-3){$\times$}
\put(70, 0){\line(1, 0){20}}
\put(91, -0.5){\dots}
\put(105, 0){\line(1, 0){20}}
\put(124, -3){$\times$}
\put(130,1){\line(1, 0){17}}
\put(130,-1){\line(1, 0){17}}
\put(143,-3){$>$}
\put(155, 0){\circle*{10}}
\put(150,-12){\mbox{\tiny$\alpha'_{m+n}$}}
\end{picture}  \\

\hline

\multirow{3}{*}{(2)}
& \begin{picture}(180, 28)(-10,-12)
\put(7,0){\circle{10}}
\put(5,-12){\tiny$\alpha_0$}
\put(11,-3){$<$}
\put (24,4){\tiny $2$}
\put(14,0){\line(1, 0){22}}
\put(37,-5){\Large$\otimes$}
\put(38,-12){\tiny$\alpha_1$}
\put(47,0){\line(1, 0){22}}
\put(65,-3){$>$}
\put(75, -0.5){\dots}
\put(91,0){\line(1, 0){20}}
\put(110, -3){$\times$}
\put(116,1){\line(1, 0){17}}
\put(116,-1){\line(1, 0){17}}
\put(130,-3){$>$}
\put(142, 0){\circle*{10}}
\put(136,-12){\mbox{\tiny$\alpha_{m+n}$}}
\end{picture}
& \begin{picture}(180, 28)(-10,-12)
\put(7,0){\circle{10}}
\put(5,-12){\tiny$\alpha'_0$}
\put(11,-3){$<$}
\put (24,4){\tiny $2$}
\put(14,0){\line(1, 0){22}}
\put(37,-5){\Large$\otimes$}
\put(38,-12){\tiny$\alpha'_1$}
\put(47,0){\line(1, 0){22}}
\put(65,-3){$>$}
\put(75, -0.5){\dots}
\put(91,0){\line(1, 0){20}}
\put(110, -3){$\times$}
\put(116,1){\line(1, 0){17}}
\put(116,-1){\line(1, 0){17}}
\put(130,-3){$>$}
\put(142, 0){\circle{10}}
\put(136,-12){\mbox{\tiny$\alpha'_{m+n}$}}
\end{picture}  \\

\cline{2-3}
& \begin{picture}(180, 28)(-10,-12)
\put(7,0){\circle{10}}
\put(5,-12){\tiny$\alpha_0$}
\put(11,-3){$<$}
\put (24,4){\tiny $2$}
\put(14,0){\line(1, 0){22}}
\put(37,-5){\Large$\otimes$}
\put(38,-12){\tiny$\alpha_1$}
\put(47,0){\line(1, 0){22}}
\put(65,-3){$>$}
\put(75, -0.5){\dots}
\put(91,0){\line(1, 0){20}}
\put(110, -3){$\times$}
\put(116,1){\line(1, 0){17}}
\put(116,-1){\line(1, 0){17}}
\put(130,-3){$>$}
\put(142, 0){\circle{10}}
\put(136,-12){\mbox{\tiny$\alpha_{m+n}$}}
\end{picture}
& \begin{picture}(180, 28)(-10,-12)
\put(7,0){\circle{10}}
\put(5,-12){\tiny$\alpha'_0$}
\put(11,-3){$<$}
\put (24,4){\tiny $2$}
\put(14,0){\line(1, 0){22}}
\put(37,-5){\Large$\otimes$}
\put(38,-12){\tiny$\alpha'_1$}
\put(47,0){\line(1, 0){22}}
\put(65,-3){$>$}
\put(75, -0.5){\dots}
\put(91,0){\line(1, 0){20}}
\put(110, -3){$\times$}
\put(116,1){\line(1, 0){17}}
\put(116,-1){\line(1, 0){17}}
\put(130,-3){$>$}
\put(142, 0){\circle*{10}}
\put(136,-12){\mbox{\tiny$\alpha'_{m+n}$}}
\end{picture}\\
\hline
\multirow{6}{*}{(3)}
&\begin{picture}(180, 60)(-26,-28)
\put(0, 15){\circle{10}}
\put(-5,23){\tiny$\alpha_0$}
\put(0, -16){\circle{10}}
\put(-4,-28){\tiny$\alpha_1$}
\put(15, -3){\line(-1, -1){10}}
\put(15, 3){\line(-1, 1){10}}
\put(16, -3){$\times$}
\put(24, 0){\line(1, 0){20}}
\put(46, -0.5){\dots}
\put(60,0){\line(1, 0){20}}
\put(79, -3){$\times$}
\put(86,1){\line(1, 0){17}}
\put(86,-1){\line(1, 0){17}}
\put(100,-3){$>$}
\put(112, 0){\circle*{10}}
\put(106,-15){\tiny$\alpha_{m+n}$}
\end{picture}
& \begin{picture}(180, 60)(-26,-28)
\put(0, 15){\circle{10}}
\put(-4,24){\tiny$\alpha'_0$}
\put(0, -16){\circle{10}}
\put(-4,-28){\tiny$\alpha'_1$}
\put(15, -3){\line(-1, -1){10}}
\put(15, 3){\line(-1, 1){10}}
\put(16, -3){$\times$}
\put(24, 0){\line(1, 0){20}}
\put(46, -0.5){\dots}
\put(60,0){\line(1, 0){20}}
\put(79, -3){$\times$}
\put(86,1){\line(1, 0){17}}
\put(86,-1){\line(1, 0){17}}
\put(100,-3){$>$}
\put(112, 0){\circle{10}}
\put(106,-15){\tiny$\alpha'_{m+n}$}
\end{picture} \\
\cline{2-3}
&\begin{picture}(180, 60)(-26,-28)
\put(0, 15){\circle{10}}
\put(-5,23){\tiny$\alpha_0$}
\put(0, -16){\circle{10}}
\put(-4,-28){\tiny$\alpha_1$}
\put(15, -3){\line(-1, -1){10}}
\put(15, 3){\line(-1, 1){10}}
\put(16, -3){$\times$}
\put(24, 0){\line(1, 0){20}}
\put(46, -0.5){\dots}
\put(60,0){\line(1, 0){20}}
\put(79, -3){$\times$}
\put(86,1){\line(1, 0){17}}
\put(86,-1){\line(1, 0){17}}
\put(100,-3){$>$}
\put(112, 0){\circle{10}}
\put(106,-15){\tiny$\alpha_{m+n}$}
\end{picture}
&\begin{picture}(180, 60)(-26,-28)
\put(0, 15){\circle{10}}
\put(-4,24){\tiny$\alpha'_0$}
\put(0, -16){\circle{10}}
\put(-4,-28){\tiny$\alpha'_1$}
\put(15, -3){\line(-1, -1){10}}
\put(15, 3){\line(-1, 1){10}}
\put(16, -3){$\times$}
\put(24, 0){\line(1, 0){20}}
\put(46, -0.5){\dots}
\put(60,0){\line(1, 0){20}}
\put(79, -3){$\times$}
\put(86,1){\line(1, 0){17}}
\put(86,-1){\line(1, 0){17}}
\put(100,-3){$>$}
\put(112, 0){\circle*{10}}
\put(106,-15){\tiny$\alpha'_{m+n}$}
\end{picture} \\
\hline
\multirow{6}{*}{(4)}
&\begin{picture}(180, 63)(-26,-26)
\put(0, 13){\Large$\otimes$}
\put(2,26){\tiny$\alpha_0$}
\put(0, -17){\Large$\otimes$}
\put(2,-26){\tiny$\alpha_1$}
\put(4,-7){\line(0, 1){20}}
\put(6,-7){\line(0, 1){20}}
\put(20, -2){\line(-1, -1){10}}
\put(20, 3){\line(-1, 1){10}}
\put(20, -3){$\times$}
\put(26, 0){\line(1, 0){20}}
\put(50, -0.5){\dots}
\put(65,0){\line(1, 0){20}}
\put(83, -3){$\times$}
\put(90,1){\line(1, 0){17}}
\put(90,-1){\line(1, 0){17}}
\put(105,-3){$>$}
\put(117, 0){\circle*{10}}
\put(110,-15){\tiny$\alpha_{m+n}$}
\end{picture}
&\begin{picture}(180, 63)(-26,-26)
\put(0, 13){\Large$\otimes$}
\put(2,26){\tiny$\alpha'_0$}
\put(0, -17){\Large$\otimes$}
\put(2,-26){\tiny$\alpha'_1$}
\put(4,-7){\line(0, 1){20}}
\put(6,-7){\line(0, 1){20}}
\put(20, -2){\line(-1, -1){10}}
\put(20, 3){\line(-1, 1){10}}
\put(20, -3){$\times$}
\put(26, 0){\line(1, 0){20}}
\put(50, -0.5){\dots}
\put(65,0){\line(1, 0){20}}
\put(83, -3){$\times$}
\put(90,1){\line(1, 0){17}}
\put(90,-1){\line(1, 0){17}}
\put(105,-3){$>$}
\put(117, 0){\circle{10}}
\put(110,-15){\tiny$\alpha'_{m+n}$}
\end{picture} \\
\cline{2-3}
&\begin{picture}(180, 63)(-26,-26)
\put(0, 13){\Large$\otimes$}
\put(2,26){\tiny$\alpha_0$}
\put(0, -17){\Large$\otimes$}
\put(2,-26){\tiny$\alpha_1$}
\put(4,-7){\line(0, 1){20}}
\put(6,-7){\line(0, 1){20}}
\put(20, -2){\line(-1, -1){10}}
\put(20, 3){\line(-1, 1){10}}
\put(20, -3){$\times$}
\put(26, 0){\line(1, 0){20}}
\put(50, -0.5){\dots}
\put(65,0){\line(1, 0){20}}
\put(83, -3){$\times$}
\put(90,1){\line(1, 0){17}}
\put(90,-1){\line(1, 0){17}}
\put(105,-3){$>$}
\put(117, 0){\circle{10}}
\put(110,-15){\tiny$\alpha_{m+n}$}
\end{picture}
&\begin{picture}(180, 63)(-26,-26)
\put(0, 13){\Large$\otimes$}
\put(2,26){\tiny$\alpha'_0$}
\put(0, -17){\Large$\otimes$}
\put(2,-26){\tiny$\alpha'_1$}
\put(4,-7){\line(0, 1){20}}
\put(6,-7){\line(0, 1){20}}
\put(20, -2){\line(-1, -1){10}}
\put(20, 3){\line(-1, 1){10}}
\put(20, -3){$\times$}
\put(26, 0){\line(1, 0){20}}
\put(50, -0.5){\dots}
\put(65,0){\line(1, 0){20}}
\put(83, -3){$\times$}
\put(90,1){\line(1, 0){17}}
\put(90,-1){\line(1, 0){17}}
\put(105,-3){$>$}
\put(117, 0){\circle*{10}}
\put(110,-15){\tiny$\alpha'_{m+n}$}
\end{picture} \\
\hline
\end{tabular}
\end{table}

\begin{proof}[Proof of Theorem \ref{thm:iso-main} for $(\g, \g')=(\Sl(2m+1|2n)^{(2)}, \osp(2n+1|2m)^{(1)})$] For $1\le i\le m+n$, we define $E_i, F_i, K_i$, $E'_i, F'_i, K'_i$   by \eqref{eq:connect B}.
Note that when the nodes of $\alpha_0$ and $\alpha'_0$ are removed,
the Dynkin diagrams in Table \ref{table:Dynkin diagram-sl2}  reduce to the Dynkin diagrams for the finite dimensional Lie superalgebras $\osp(2m+1|2n)$ and $\osp(2n+1|2m)$.
Thus by \eqref{eq:connect B}, the same reasoning in Section \ref{sect:osp} can show
that the elements
$E_i, F_i, K_i$ (resp. $E'_i, F'_i, K'_i$)  for $1\le i\le m+n$ have
the desired properties.

What remains to be done, in order to complete the proof of
Theorem \ref{thm:iso-main},   is to construct elements
$E_0, F_0, K_0\in \UU_q(\g, \Pi)$ (resp. $E'_0, F'_0, K'_0\in \UU_t(\g', \Pi')$),
which satisfy the commutation relations obeyed by $e'_0, f'_0, k'_0$ (resp $e_0, f_0, k_0$).  We do this for each of the four types of diagrams in
Table \ref{table:Dynkin diagram-sl2}.
The proofs for $E_0, F_0, K_0$  and  for $E'_0, F'_0, K'_0$ are similar to those in  Section \ref{sect:osp}, thus we will give the constructions for these elements only.

\begin{case}
{\em Type (1) and  Type (2)  Dynkin diagrams  in Table \ref{table:Dynkin diagram-sl2}}.
\begin{equation}\label{eq:connect AB-case1}
\begin{aligned}
&\quad E_0=e_0,\quad F_0=f_0, \quad K_0=k_0, \text{ and,}\ \ E'_0=e'_0,\quad F'_0=f'_0, \quad K'_0=k'_0.
\end{aligned}
\end{equation}
\end{case}
Note that for Type (2) Dynkin diagrams with $m+n>2$, the higher order Serre relations  involving $e_0$ or $f_0$ are either of type (D) (for $2\in\tau$) or (E) (for $2\notin\tau$)  (see Definition \ref{defi:quantised}).  

\begin{case}
{\em Type (3) Dynkin diagrams in Table \ref{table:Dynkin diagram-sl2}}.
In this case,
\begin{eqnarray}\label{eq:connect AB-case3}
\begin{aligned}
&E_0=\Phi_{2}e_0,\quad F_0=\Phi_1 f_0, \quad K_0=\sigma_1k_0,\\
&E'_0=\Phi'_{2}e'_0,\quad F'_0=\Phi'_1 f'_0, \quad K'_0=\sigma'_1k'_0.
\end{aligned}
\end{eqnarray}
\end{case}

\begin{case}
{\em Type (4) Dynkin diagrams in Table \ref{table:Dynkin diagram-sl2}}.
This time $\alpha_0$ is an odd simple root. Define
\begin{eqnarray}\label{eq:connect AB-case4}
\begin{aligned}
&E_0=\tilde{\Phi}_{3}e_0,\quad \ \ F_0= \tilde{\Phi}_1 f_0,\quad K_0=\sigma_1k_0,\\
&E'_0=\tilde{\Phi}'_{3}e'_0,\quad \ \ F'_0= \tilde{\Phi}'_1 f'_0,\quad K'_0=\sigma'_1k'_0.
\end{aligned}
\end{eqnarray}
\end{case}
\end{proof}

\subsection{The case of  $\osp(2m+2|2n)^{(2)}$ and $\osp(2n+2|2m)^{(2)}$}\label{sect:osp2}
\setcounter{case}{0}
The Dynkin diagrams of $\g=\osp(2m+2|2n)^{(2)}$ and $\g'=\osp(2n+2|2m)^{(2)}$ are  given in Table \ref{table:Dynkin diagram-osp2},
where the diagrams for a fundamental system $\Pi$ of $\g$ and the corresponding fundamental system  $\Pi'$ of $\g'$ are shown in the same row.

\begin{table}[!htbp]
\caption{Dynkin diagrams of $\osp(2m+2|2n)^{(2)}$ and $\osp(2n+2|2m)^{(2)}$}
\label{table:Dynkin diagram-osp2}
\begin{tabular}{  >{\centering\arraybackslash}m{0.4in} |>{\centering\arraybackslash}m{2.2in}|   >{\centering\arraybackslash}m{2.2in}  }
\hline
Type
&\vspace{3mm} $\g=\osp(2m+2|2n)^{(2)}$ \vspace{2mm} & $\g'=\osp(2n+2|2m)^{(2)}$\\
\hline
\multirow{3}{*}{(1)}
&\begin{picture}(150, 30)(-6,-14)
\put(7,0){\circle{10}}
\put(3,-12){\tiny $\alpha_0$}
\put(16,1){\line(1, 0){18}}
\put(16,-1){\line(1, 0){18}}
\put(12,-3){$<$}
\put(34, -3){$\times$}
\put(40, 0){\line(1, 0){20}}
\put(61, -0.5){\dots}
\put(75, 0){\line(1, 0){20}}
\put(94, -3){$\times$}
\put(100,1){\line(1, 0){17}}
\put(100,-1){\line(1, 0){17}}
\put(113,-3){$>$}
\put(125, 0){\circle*{10}}
\put(115,-12){ \tiny$\alpha_{m+n}$}
\end{picture}
&\begin{picture}(150, 30)(-10,-14)
\put(7,0){\circle*{10}}
\put(3,-12){\tiny $\alpha'_0$}
\put(16,1){\line(1, 0){18}}
\put(16,-1){\line(1, 0){18}}
\put(12,-3){$<$}
\put(34, -3){$\times$}
\put(40, 0){\line(1, 0){20}}
\put(61, -0.5){\dots}
\put(75, 0){\line(1, 0){20}}
\put(94, -3){$\times$}
\put(100,1){\line(1, 0){17}}
\put(100,-1){\line(1, 0){17}}
\put(113,-3){$>$}
\put(125, 0){\circle{10}}
\put(115,-12){ \tiny$\alpha'_{m+n}$}
\end{picture} \\
\cline{2-3}
&\begin{picture}(150, 30)(-6,-14)
\put(7,0){\circle{10}}
\put(3,-12){\tiny $\alpha_0$}
\put(16,1){\line(1, 0){18}}
\put(16,-1){\line(1, 0){18}}
\put(12,-3){$<$}
\put(34, -3){$\times$}
\put(40, 0){\line(1, 0){20}}
\put(61, -0.5){\dots}
\put(75, 0){\line(1, 0){20}}
\put(94, -3){$\times$}
\put(100,1){\line(1, 0){17}}
\put(100,-1){\line(1, 0){17}}
\put(113,-3){$>$}
\put(125, 0){\circle{10}}
\put(115,-12){ \tiny$\alpha_{m+n}$}
\end{picture}
&\begin{picture}(150, 30)(-10,-14)
\put(7,0){\circle*{10}}
\put(3,-12){\tiny $\alpha'_0$}
\put(16,1){\line(1, 0){18}}
\put(16,-1){\line(1, 0){18}}
\put(12,-3){$<$}
\put(34, -3){$\times$}
\put(40, 0){\line(1, 0){20}}
\put(61, -0.5){\dots}
\put(75, 0){\line(1, 0){20}}
\put(94, -3){$\times$}
\put(100,1){\line(1, 0){17}}
\put(100,-1){\line(1, 0){17}}
\put(113,-3){$>$}
\put(125, 0){\circle*{10}}
\put(115,-12){ \tiny$\alpha'_{m+n}$}
\end{picture} \\
\hline
\multirow{3}{*}{(2)}
&\begin{picture}(150, 30)(-6,-14)
\put(7,0){\circle*{10}}
\put(3,-12){\tiny $\alpha_0$}
\put(16,1){\line(1, 0){18}}
\put(16,-1){\line(1, 0){18}}
\put(12,-3){$<$}
\put(34, -3){$\times$}
\put(40, 0){\line(1, 0){20}}
\put(61, -0.5){\dots}
\put(75, 0){\line(1, 0){20}}
\put(94, -3){$\times$}
\put(100,1){\line(1, 0){17}}
\put(100,-1){\line(1, 0){17}}
\put(113,-3){$>$}
\put(125, 0){\circle*{10}}
\put(115,-12){ \tiny$\alpha_{m+n}$}
\end{picture}
&\begin{picture}(150, 30)(-10,-14)
\put(7,0){\circle{10}}
\put(3,-12){\tiny $\alpha'_0$}
\put(16,1){\line(1, 0){18}}
\put(16,-1){\line(1, 0){18}}
\put(12,-3){$<$}
\put(34, -3){$\times$}
\put(40, 0){\line(1, 0){20}}
\put(61, -0.5){\dots}
\put(75, 0){\line(1, 0){20}}
\put(94, -3){$\times$}
\put(100,1){\line(1, 0){17}}
\put(100,-1){\line(1, 0){17}}
\put(113,-3){$>$}
\put(125, 0){\circle{10}}
\put(115,-12){ \tiny$\alpha'_{m+n}$}
\end{picture} \\
\cline{2-3}
&\begin{picture}(150, 30)(-6,-14)
\put(7,0){\circle*{10}}
\put(3,-12){\tiny $\alpha_0$}
\put(16,1){\line(1, 0){18}}
\put(16,-1){\line(1, 0){18}}
\put(12,-3){$<$}
\put(34, -3){$\times$}
\put(40, 0){\line(1, 0){20}}
\put(61, -0.5){\dots}
\put(75, 0){\line(1, 0){20}}
\put(94, -3){$\times$}
\put(100,1){\line(1, 0){17}}
\put(100,-1){\line(1, 0){17}}
\put(113,-3){$>$}
\put(125, 0){\circle{10}}
\put(115,-12){ \tiny$\alpha_{m+n}$}
\end{picture}
&\begin{picture}(150, 30)(-10,-14)
\put(7,0){\circle{10}}
\put(3,-12){\tiny $\alpha'_0$}
\put(16,1){\line(1, 0){18}}
\put(16,-1){\line(1, 0){18}}
\put(12,-3){$<$}
\put(34, -3){$\times$}
\put(40, 0){\line(1, 0){20}}
\put(61, -0.5){\dots}
\put(75, 0){\line(1, 0){20}}
\put(94, -3){$\times$}
\put(100,1){\line(1, 0){17}}
\put(100,-1){\line(1, 0){17}}
\put(113,-3){$>$}
\put(125, 0){\circle*{10}}
\put(115,-12){ \tiny$\alpha'_{m+n}$}
\end{picture} \\
\hline
\end{tabular}
\end{table}

\begin{proof}[Proof of Theorem \ref{thm:iso-main}  for $(\g, \g')=
 (\osp(2m+2|2n)^{(2)}, \osp(2n+2|2m)^{(2)})$]
We will merely construct the elements $E_i, F_i, K_i, E'_i, F'_i, K'_i$ here, as
the proof of Theorem \ref{thm:iso-main} is much the same as in the previous  cases.
For $1\leq i\leq m+n$, the elements $E_i, F_i, K_i, E'_i, F'_i, K'_i$ are given by
\eqref{eq:connect B}; and for $i=0$, they are defined as follows.
\begin{case}
{\em Type (1) Dynkin diagrams in Table \ref{table:Dynkin diagram-osp2}}.
\begin{eqnarray}\label{eq:connect D-case1}
\begin{aligned}
&E_0=\tilde{\Phi}_2\prod_{j\in\tau}(\tilde{\Phi}_1\tilde{\Phi}_{j+1})\cdot e_0,\quad F_0=\tilde{\Phi}_1\prod_{j\in\tau}(\tilde{\Phi}_1\tilde{\Phi}_{j+1})\cdot f_0,\quad K_0=\Phi_1\cdot k_0;\\
&E'_0=\tilde{\Phi}'_2\prod_{j\in\tau'}(\tilde{\Phi}'_1\tilde{\Phi}'_{j+1})\cdot e'_0,\quad F'_0= \tilde{\Phi}'_1\prod_{j\in\tau'}(\tilde{\Phi}'_1\tilde{\Phi}'_{j+1})\cdot f'_0, \quad K'_0=\Phi'_1\cdot k'_0.
\end{aligned}
\end{eqnarray}
In this case, $0\notin\tau, m+n\in\tau$ while $0\in\tau', m+n\notin\tau'$.
\end{case}

\begin{case}
{\em Type (2) Dynkin diagrams in Table \ref{table:Dynkin diagram-osp2}}.
\begin{equation}\label{eq:connect D-case2}
\begin{aligned}
&E_0=\Phi_1\prod_{j\in\tau}(\tilde{\Phi}_1\tilde{\Phi}_{j+1})\cdot e_0,\quad F_0=\prod_{j\in\tau}(\tilde{\Phi}_1\tilde{\Phi}_{j+1})\cdot f_0,\quad K_0=\Phi_1\cdot k_0;\\
&E'_0=\tilde{\Phi}'_1\prod_{j\in\tau'}(\tilde{\Phi}'_1\tilde{\Phi}'_{j+1})\cdot e'_0,\quad F'_0=\prod_{j\in\tau'}(\tilde{\Phi}'_1\tilde{\Phi}'_{j+1})\cdot f'_0,\quad K'_0=\Phi'_1\cdot k'_0.
\end{aligned}
\end{equation}
In this case, $0\in\tau$ and $0\notin\tau'$.
\end{case}
\end{proof}

\section{Hopf superalgebra isomorphisms}\label{sect:tensor-cats}

We will prove Theorem \ref{them:hopf-connect} in this section. We begin by discussing some facts on Hopf superalgebras, which are needed presently,
but are not expected to be widely known.
\subsection{Picture changes and Drinfeld twists for Hopf superalgebras}\label{sect:Hopf}
\subsubsection{Picture changes}\label{sect:picture}

The category of vector superspaces can be regarded as the category of representations of the group algebra of $\Z_2:=\{1, u\}$ where $u^2=1$, which is a triangular Hopf algebra with the universal $R$-matrix
\[
R :=\frac{1}{2}\left(1 \otimes 1 + 1  \otimes u + u  \otimes 1- u  \otimes u\right) \in \C[\Z_2]  \otimes\C[\Z_2].
\]
A Hopf superalgebra $\cH$ is then a Hopf algebra in this category. The grading of $\cH$ is given by the $\Z_2$-action such that
\[
u. a = (-1)^{[a]} a,
\]
for any homogeneous $a\in \cH$. For any  $a, b\in \cH$, if we write their co-products as
$\Delta(a) = \sum a_{(1)}\otimes a_{(2)}$ and $\Delta(b) = \sum b_{(1)}\otimes b_{(2)}$ respectively, then $\Delta(a b)$ is given by
\[
\begin{aligned}
\Delta(a b) &= (m\otimes m)\left(\sum a_{(1)}\otimes \tau R(a_{(2)}\otimes b_{(1)})\otimes b_{(2)}\right),
\end{aligned}
\]
where $m$ is the multiplication of $\cH$, and $\tau: v\otimes w\mapsto w\otimes v$ is the usual permutation map (without signs).  Then clearly
\[
\begin{aligned}
\Delta(a b)
&= \sum (-1)^{[b_{(1)})][a_{(2)}]}a_{(1)}b_{(1)}\otimes a_{(2)}b_{(2)}.
\end{aligned}
\]

By changing the category of $\Z_2$-representations one obtains a non-isomorphic Hopf superalgebra from any given one,
such that its category of representations is equivalent to that of the original Hopf superalgebra as tensor category, see
\cite[Theorem 3.1.1]{AEG}  and \cite[ Chapter 10.1]{Ma}.
 Let us describe this more explicitly.

{\bf PC1} (\cite[Theorem 3.1.1]{AEG}).
Let $(H, \Delta, \epsilon, S)$ be an ordinary Hopf algebra with a group like element $u$ such that $u^2 = 1$. Using $u$, we decompose $H$ as a vector space into
$H=H_0\oplus H_1$ with
\begin{eqnarray}\label{eq:PO}
H_i=\left\{x\in H\mid u x u^{-1} = (-1)^i x\right\}.
\end{eqnarray}
This clearly defines a $\Z_2$-grading for $H$ as an associative algebra, thus turning it into an superalgebra.  We set $[x]=i$ for $x\in H_i$.
For any $x\in H$, write $\Delta(x)=\Delta_0(x) + \Delta_1(x)$ with  $\Delta_0(x) \in H\otimes H_0$ and $\Delta_1(x)\in H\otimes H_1$.  Define maps
\begin{eqnarray}\label{eq:corresp}
\begin{aligned}
&\Delta_u: H\longrightarrow H\otimes H, &&\quad \Delta_u(x)=\Delta_0(x) + \Delta_1(x)(u\otimes 1), \\
&S_u: H\longrightarrow H, &&\quad S_u(x) = u^{[x]} S(x).
\end{aligned}
\end{eqnarray}
Then $(H, \Delta_u, \epsilon, S_u)$ is a Hopf superalgebra. The
element $u$ acts as the parity operator (PO) of this Hopf superalgebra in the sense of \eqref{eq:PO}.

{\bf PC2} (\cite[Theorem 3.1.1]{AEG}).
Let $(\cH, \Delta, \epsilon, S)$ be a Hopf superalgebra with a group like element $g$ satisfying $g^2=1$, which acts as the parity operator in the sense that
$g xg^{-1} =(-1)^{[x]} x$ for all homogeneous $x\in \cH$. We define maps
$\Delta_g: \cH\longrightarrow \cH\otimes \cH$ and $S_g: \cH\longrightarrow \cH$ in exactly the same way as in \eqref{eq:corresp}. Then $(\cH, \Delta_g, \epsilon, S_g)$ is an ordinary Hopf algebra.

{\bf PC}. Let $(\cH, \Delta, \epsilon, S)$ be a Hopf superalgebra. Suppose that it has two group like elements $g$ and $u$  such that
\[
g^2=1=u^2, \quad g u=u g,\quad \text{and $g$ acts as the parity operator}.
\]
We apply {\bf PC2} to obtain an ordinary Hopf algebra, and then apply {\bf PC1} with $u$ to the ordinary Hopf algebra to obtain a new Hopf superalgebra with parity operator $u$:
\begin{eqnarray*}
\begin{array}{c}
\text{Hopf superalgebra}\\
(\cH, \Delta, \epsilon, S)\\
\text{with PO $g$}
\end{array}
\stackrel{PC2}{\longsquiggly}
\begin{array}{c}
\text{Hopf algebra}\\
(\cH, \Delta_g, \epsilon, S_g)\\
\text{with $u$}
\end{array}
\stackrel{PC1}{\longsquiggly}
\begin{array}{c}
\text{Hopf superalgebra}\\
(\cH, (\Delta_g)_u, \epsilon, (S_g)_u)\\
\text{with PO $u$}.
\end{array}
\end{eqnarray*}

\begin{definition} \label{def:PC} Call the operation of constructing the new Hopf superalgebra $(\cH, (\Delta_g)_u, \epsilon, (S_g)_u)$ with parity operator $u$ from a given Hopf superalgebra
$(\cH, \Delta, \epsilon, S)$ with parity operator $g$  a {\em picture change} ({\bf PC}) with respect to $g$ and $u$.
\end{definition}

\begin{remark}\label{rem:bosonisation}
This is loosely called ``bosonisation'' in the literature (see \cite{Ma} in particular).
As bosonisation means something very different in quantum field theory, we prefer the term ``picture change''.
\end{remark}

%
%

Representation categories of Hopf algebras and Hopf superalgebras are strict tensor categories.
For any $\cH$-modules $M$ and $N$,  the $\Z_2$-graded action of $\Delta(x)$ and  the ordinary action of $\Delta_g(x)$ on $M\otimes N$ coincide for all $x\in \cH$.
This in essence implies the tensor equivalence of the representation categories of the Hopf
superalgebra $(\cH, \Delta, \epsilon, S)$ and the ordinary Hopf algebra $(\cH, \Delta_g, \epsilon, S_g)$ related by {\bf PC2}.  Similarly one can show the
tensor equivalence of the representation categories of the Hopf
algebra $(H, \Delta, \epsilon, S)$ and Hopf superalgebra $(H, \Delta_u, \epsilon, S_u)$
related by {\bf PC1}. See  \cite[Theorem 3.1.1]{AEG}.

We summarise the above into the following

\begin{theorem} \label{thm:PC} Let $(\cH, \Delta, \epsilon, S)$ be a Hopf superalgebra with group like elements $g$ and $u$ as described above such that $g$ acts as the parity operator.
Then a picture change turns this Hopf superalgebra into a new Hopf superalgebra
$(\cH, (\Delta_g)_u, \epsilon, (S_g)_u)$ with parity operator $u$. The categories of representations of the two Hopf superalgebras are equivalent as strict tensor categories.
\end{theorem}

\subsubsection{Twisting the coalgebra structure}\label{sect:twisting}
\setcounter{case}{0}

Recall the following fact \cite{D2, R}. Let $(\cH, \Delta, \epsilon, S)$ be a Hopf superalgebra. Given an invertible even element $\cJ\in \cH\otimes\cH$ satisfying the conditions
\begin{eqnarray}\label{eq:twist}
\begin{aligned}
&(\Delta\otimes\id)(\cJ)(\cJ\otimes 1)=(\id\otimes\Delta)(\cJ)(1\otimes\cJ), \\
&(\epsilon\otimes\id)(\cJ)=(\id\otimes\epsilon)(\cJ)=1,
\end{aligned}
\end{eqnarray}
one can twist the coalgebra structure to obtain a new Hopf superalgebra $(\cH, \Delta^{\cJ}, \epsilon, S^{\cJ})$ with the same underlying associative superalgebraic structure on $\cH$.
The new comultiplication $\Delta^{\cJ}$ and antipode $S^{\cJ}$ are given by
\[
\Delta^\cJ(x)= \cJ^{-1}\Delta(x) \cJ, \quad S^\cJ(x)=\cG^{-1} S(x) \cG, \quad \forall x\in \cH,
\]
with $\cG=m\circ(S\otimes\id)(\cJ)$, where $m$ is the multiplication of $\cH$. The element $\cJ$ is called a {\em Drinfeld twist} for $\cH$. Note that twisting does not change the counit.

\subsection{Quantum correspondences}
 Keep the notation in Section \ref{sect:quantum}.
Let $\g$ be a Lie superalgebra or affine Lie superalgebra in Table \ref{table:classical} or Table \ref{table:affine} with a fundamental system $\Pi$.
Then there exists a corresponding $\g'$ such that $(\g, \g')$ is a pair in Theorem \ref{thm:main-quan}.  Now $\Pi'=\phi(\Pi)$ is a fundamental system of $\g'$.

Consider $\UU_q(\g, \Pi)$ as a Hopf superalgebra with the standard grading. As before, we denote its comultiplication, counit and antipode by $\Delta, \epsilon$ and $S$ respectively. Let
\begin{eqnarray}\label{eq:u1u2}
\begin{aligned}
&u_1:=\prod_{i\in\tau}(\tilde{\Phi}_1\tilde{\Phi}_{i+1}), \quad &u_2:=\tilde{\Phi}_1 u_1,
\end{aligned}
\end{eqnarray}
Note that $\tilde{\Phi}_1\tilde{\Phi}_{i+1} X^\pm_j (\tilde{\Phi}_1\tilde{\Phi}_{i+1})^{-1} = (-1)^{\delta_{i j}} X^\pm_j$ for all $i, j$, where $X^+_j=e_j$ and $X^-_j=f_j$. [Recall that $\tilde{\Phi}_k=1$ if $k>m+n$ by convention.] Thus
$u_1$ is the parity operator of $\UU_q(\g, \Pi)$.

Applying a picture change with respect to $u_1$ and $u_2$ to $(\UU_q(\g, \Pi), \Delta, \epsilon, S)$, we obtain the
Hopf superalgebra $(\UU_q(\g, \Pi), (\Delta_{u_1})_{u_2}, \epsilon, (S_{u_1})_{u_2})$
with parity operator $u_2$.
The new $\Z_2$-grading of $\UU_q(\g, \Pi)$, induced by $u_2$,  is given by
\begin{eqnarray}\label{eq:new-grade}
\begin{aligned}
&\UU_q(\g, \Pi)=\UU_q(\g, \Pi)'_0\oplus\UU_q(\g, \Pi)'_1  \quad \text{with}\\
&\UU_q(\g, \Pi)'_\theta=\left\{x\in \UU_q(\g, \Pi)\mid u_2 x u_2^{-1} = (-1)^\theta x\right\},
\quad \theta=0, 1.
\end{aligned}
\end{eqnarray}
Write $\tilde{\Delta}=(\Delta_{u_1})_{u_2}$ and $\tilde{S}=(S_{u_1})_{u_2}$, and use
$(\UU_q(\g, \Pi), \tilde{\Delta}, \epsilon, \tilde{S})$ to denote
this new Hopf superalgebra with the $\Z_2$-grading given by \eqref{eq:new-grade}.

Recall the elements $E_i, F_i, K_i^{\pm 1}$ of $\UU_q(\g, \Pi)$ introduced in Section \ref{sect:quantum}.
They together with the elements $\sigma_i$ generate $\UU_q(\g, \Pi)$. We have the following easy observation.
\begin{lemma} For any fixed $i$, the elements $E_i, F_i$ belong to $\UU_q(\g, \Pi)'_0$
(resp. $\UU_q(\g, \Pi)'_1$) if and only if $\phi(\alpha_i)$ is an even (resp. odd) simple root in $\Pi'$.
\end{lemma}
This immediately implies
\begin{corollary} \label{cor:alg-iso} The associative algebra isomorphism $\UU_{-q}(\g', \Pi')\stackrel{\cong}{\longrightarrow} \UU_q(\g, \Pi)$ of Theorem \ref{thm:main-quan} defined by \eqref{eq:B-map}
is an isomorphism of superalgebras if
$\UU_q(\g, \Pi)$ is given the $\Z_2$-grading \eqref{eq:new-grade} induced by $u_2$, while $\UU_{-q}(\g', \Pi')$ has  the usual $\Z_2$-grading.
\end{corollary}

Recall that $|\Pi|$ denotes the cardinality of $\Pi$. Define
\begin{eqnarray}\label{eq:J}
\cJ:=\frac{1}{2^{|\Pi|}}\cT, \quad \cT:=\cT^{(0)}\cT^{(1)},
\end{eqnarray}
where
\[
\begin{aligned}
\cT^{(0)}:=&\prod_{i\notin\tau}\cT^{(0)}_i, \quad \cT^{(1)}:=\prod_{i\in\tau}\cT^{(1)}_i, \\
\cT^{(0)}_i:=&(1+\tilde{\Phi}_1\tilde{\Phi}_{i+1})\otimes 1+(1-\tilde{\Phi}_1\tilde{\Phi}_{i+1})\otimes \Phi_{i+1},\quad i\notin\tau, \\
\cT^{(1)}_i:=&(1+\tilde{\Phi}_1\tilde{\Phi}_{i+1})\otimes 1+(1-\tilde{\Phi}_1\tilde{\Phi}_{i+1})\otimes \tilde{\Phi}_{i+2}, \quad i\in\tau.
\end{aligned}
\]
\begin{lemma}\label{lem:j-properties}
The element $\cJ$ defined by \eqref{eq:J} satisfies the relations
 \[
 \begin{aligned}
&(\tilde{\Delta}\otimes\id)(\cJ)(\cJ\otimes 1)=(\id\otimes\tilde{\Delta})(\cJ)(1\otimes\cJ), \\
&(\epsilon\otimes\id)(\cJ)=(\id\otimes\epsilon)(\cJ)=1.
\end{aligned}
\]
\end{lemma}
\begin{proof}
The second relation is clear since $\epsilon(\Phi_i)=\epsilon(\tilde\Phi_i)=1$ for all $i$.

To prove the first relation, note that $\cJ$, $u_1$ and $u_2$ involve only the even elements $\sigma_i$ of $\UU_q(\g, \Pi)$, which commute  among themselves.
Thus the first relation is equivalent to that obtained by replacing $\tilde{\Delta}$ by $\Delta$.

For any elements $\sigma,\sigma'$ in $\mathrm{G}$, denote $x=(1+\sigma)\otimes 1+(1-\sigma)\otimes \sigma'$. It can be proven by direct computations that
$(\Delta\otimes\id)(x)(x\otimes 1)=(\id\otimes\Delta)(x)(1\otimes x)$, and hence $(\tilde{\Delta}\otimes\id)(x)(x\otimes 1)=(\id\otimes\tilde{\Delta})(x)(1\otimes x)$. As $\cT$ is the product of elements of the form $x$, this immediately leads to the first relation.
\end{proof}
\begin{remark}
We have $\cJ^{-1}=\cJ$ because $x^2=4$ for the $x$ in the proof of the above lemma.
\end{remark}

By Lemma \ref{lem:j-properties}, we can twist the Hopf superalgebra $(\UU_q(\g, \Pi), \tilde{\Delta}, \epsilon, \tilde{S})$ using the element $\cJ$ given in \eqref{eq:J} to obtain a new Hopf superalgebra $(\UU_q(\g, \Pi), \tilde{\Delta}^\cJ, \epsilon, \tilde{S}^\cJ)$. We emphasize that the $\Z_2$-grading is given by \eqref{eq:new-grade}.

\begin{lemma}\label{lem:tilde-Delta}
The comultiplication, cunit and antipode of the Hopf superalebra $(\UU_q(\g, \Pi), \tilde{\Delta}^\cJ, \epsilon, \tilde{S}^\cJ)$ are given by
\[
\begin{aligned}
&\tilde{\Delta}^\cJ(\sigma_i)=\sigma_i\otimes\sigma_i, \quad \tilde{\Delta}^\cJ(K_i)=K_i\otimes K_i, \\
&\tilde{\Delta}^\cJ(E_i)=E_i\otimes 1 + K_i  \otimes E_i, \quad \tilde{\Delta}^\cJ(F_i)=F_i\otimes K^{-1}_i + 1 \otimes F_i, \\
&\epsilon(E_i)=0, \quad \epsilon(F_i)=0, \quad \epsilon(K_i)=1,   \quad \epsilon(\sigma_i)=1\\
&\tilde{S}^\cJ(E_i)=- K^{-1}_i E_i, \quad \tilde{S}^\cJ(F_i)=- F_i K_i , \quad \tilde{S}^\cJ(K_i)=K^{-1}_i, \quad \tilde{S}^\cJ(\sigma_i)=\sigma^{-1}_i, \quad \forall i.
\end{aligned}
\]
\end{lemma}
\begin{proof} The relations for the counit are clear,
and the antipode relations can be easily obtained from the comultiplication
and the counit. Note that  $\cJ$ depends only on $\sigma_i$.
Since $K_i$ and $\sigma_i$ are all even and
commute among themselves, we immediately have
\[
\begin{aligned}
&\tilde{\Delta}^\cJ(\sigma_i)=\tilde{\Delta}(\sigma_i)=\Delta(\sigma_i)=\sigma_i\otimes\sigma_i, \quad
&\tilde{\Delta}^\cJ(K_i)=\tilde{\Delta}(K_i)=\Delta(K_i)=K_i\otimes K_i.
\end{aligned}
\]
Thus what remains to be proven are the formulae for $\tilde{\Delta}^\cJ(E_i)$
and $\tilde{\Delta}^\cJ(F_i)$. In the Hopf superalgebra $(\UU_q(\g,\Pi), \td,\ve,\ts)$,  we have, for $i,j>0$,
\[\begin{aligned}
&\td(E_i)=E_i\otimes\Phi_{i+1}+\Phi_iK_i\otimes E_i, \quad m+n\ne i\notin\tau,\\
&\td(E_j)=E_j\otimes\tilde{\Phi}_{j+2}+\tilde{\Phi}_1\tilde{\Phi}_jK_j\otimes E_j,\quad m+n\ne j\in\tau\\
&\td(E_{m+n})=E_{m+n}\otimes 1+u\Phi_{m+n}K_{m+n}\otimes E_{m+n},
\end{aligned}\]
where  $u=\tilde{\Phi}_1\prod_{j\in\tau,j\neq m+n}(\tilde{\Phi}_1\tilde{\Phi}_{j+1})$, which is $u_1$ if $m+n\in\tau$, and is $u_2$ if $m+n\notin\tau$, and
\begin{itemize}
\item[(\romannumeral1)]
for the Dynkin diagrams of Types (1), (2) and (3)  in Table \ref{table:Dynkin diagram-sl2},
\[\begin{aligned}
&\quad\td(E_0)=E_0\otimes \Phi_2+\Phi_1K_0\otimes E_0;
\end{aligned}\]
\item[(\romannumeral2)]for type (4) Dynkin diagrams in Table \ref{table:Dynkin diagram-sl2},
\[\begin{aligned}
&\td(E_0)=E_0\otimes \tilde{\Phi}_3+ K_0\otimes E_0;
\end{aligned}\]
\item[(\romannumeral3)]for all  Dynkin diagrams in Table \ref{table:Dynkin diagram-osp2},
\[\begin{aligned}
&\td(E_0)=E_0\otimes \tilde{\Phi}_2\prod_{j\in\tau}(\tilde{\Phi}_1\tilde{\Phi}_{j+1})+K_0\otimes E_0,
\end{aligned}\]
\end{itemize}
Using the above formulae, we can easily show that
\[\begin{aligned}
\td^{\cJ}(E_i)=E_i\otimes 1+ K_i\otimes E_i, \quad \forall i.
\end{aligned}\]
We can similarly prove the formula for $\td^{\cJ}(F_i)$.
\end{proof}

\begin{proof}[Proof of Theorem \ref{them:hopf-connect}]
By Corollary \ref{cor:alg-iso}, the map \eqref{eq:B-map}
is an isomorphism of associative superalgebras, and by Lemma \ref{lem:tilde-Delta}, it is a Hopf superalgebra map. Hence follows the theorem.
\end{proof}

\begin{proof}[Proof of Theorem \ref{thm:tensor-equiv}]
This easily follows from Theorem \ref{them:hopf-connect} by using
Theorem \ref{thm:PC}.
\end{proof}

\begin{remark}
We expect that for any pair $(\g, \g')$ in Theorem  \ref{thm:main-quan}, the
representation categories of $\UU_q(\g, \Pi)$ and of $\UU_{-q}(\g', \Pi')$ are
equivalent as braided strict tensor categories. One should be able to prove
this following \cite[Chapter 10.1]{Ma}.
\end{remark}

\begin{remark} Another possible approach to the proof of
Theorem \ref{them:hopf-connect} is to consider Hopf superalgebras in the category of Yetter-Drinfeld modules. The Hopf superalgebras $\UU_{q}(\g, \Pi)$ and $\UU_{-q}(\g', \Pi')$ are then expected to be quantum doubles of the same Nichols algebra of diagonal type \cite{H}. However, such a proof will necessarily be much more involved.
\end{remark}

\section*{Acknowledgements}
This research was supported by National Natural Science Foundation of China Grants No. 11301130,  No. 11431010;
and Australian Research Council Discovery-Project Grant DP140103239.
Xu wishes to thank the School of Mathematics and Statistics at the University of Sydney for its hospitality.

\end{document}